\documentclass[a4paper,10pt]{article}
\usepackage[top=2.54cm,bottom=2.0cm,left=2.0cm,right=2.54cm, includeheadfoot]{geometry}

\usepackage[T1]{fontenc}
\usepackage[utf8]{inputenc}
\usepackage[english]{babel}
\usepackage{enumerate}
\usepackage{multirow,booktabs}
\usepackage[table]{xcolor}
\usepackage{fullpage}
\usepackage{lastpage}
\usepackage{indentfirst}

\usepackage{footnote}

\usepackage{amsmath,amsfonts,amssymb,amscd,amsthm}

\usepackage{bm}

\usepackage[all,2cell]{xy} \UseAllTwocells \SilentMatrices

\usepackage{hyperref}

\usepackage{subfiles}

\usepackage{multicol}

\usepackage{lineno}

\newtheorem{teo}{Theorem}[section]
\newtheorem{lem}[teo]{Lemma} 
\newtheorem{cor}[teo]{Corollary}

\newtheorem{defn}[teo]{Definition} 
\newtheorem{ex}[teo]{Example}

\newtheorem*{claim*}{Claim}

\newtheorem{rem}[teo]{Remark}

\usepackage{filecontents}

\begin{document}

\title{On algebraic extensions and algebraic closures of superfields}

\author{Kaique Matias de Andrade Roberto\thanks{Institute of Mathematics and Statistics, University of São Paulo, Brazil. kaique.roberto@usp.br} \\  Hugo Rafael de Oliveira Ribeiro\thanks{Institute of Mathematics and Statistics, University of São Paulo, Brazil. hugor@ime.usp.br} \\
Hugo Luiz Mariano\thanks{Institute of Mathematics and Statistics, University of São Paulo, Brazil. hugomar@ime.usp.br}}

\maketitle

\begin{abstract}
    Building over recent results, we expand the basic theory of algebraic extensions to the realm of superfields -a field with multivalued sum and product-, showing that every superfield has a (unique up to isomorphism) full algebraic extension to a superfield that is algebraically closed. Moreover we show that  every infinite algebraically closed superfield admits quantifier elimination procedure.
\end{abstract}

\section{Introduction}

The concept of multialgebraic structure -- an ``algebraic like'' structure but endowed with  multiple valued operations -- has 
been studied since the 1930's; in particular, the concept of hyperrings was introduced by Krasner in the 1950's. Some general 
algebraic study has been made on multialgebras: see for instance \cite{golzio2018brief} and \cite{pelea2006multialgebras}.

Since the middle of the 2000s decade, the notion of multiring have obtained more attention: a multiring is a lax hyperring, satisfying an weak distributive law, but  hyperfields and multifields coincide. Multirings   has been studied for  applications many areas: in abstract quadratic forms theory 
(\cite{marshall2006real}, \cite{worytkiewiczwitt2020witt}, \cite{roberto2021quadratic}), tropical geometry (\cite{viro2010hyperfields}, \cite{jun2015algebraic}), algebraic geometry ((\cite{jun2021geometry}, \cite{baker2021descartes}), valuation theory (\cite{jun2018valuations}), Hopf algebras (\cite{eppolito2020hopf}), etc (\cite{baker2021structure}, \cite{ameri2020advanced}, \cite{ameri2017multiplicative}, \cite{bowler2021classification}).

A more detailed account of 
variants of concept of polynomials over hyperrings is even more recent (\cite{jun2015algebraic}, \cite{ameri2019superring}, \cite{baker2021descartes}). In \cite{roberto2021superrings}
 we start a model-theoretic oriented analysis of multialgebras  introducing the class of algebraically closed hyperfields and 
providing variant proof of quantifier elimination flavor,   based on new results on  superring of polynomials.

In the present work we provide new steps the program of studying the hyperfields (and natural variants: superfields) under a natural notion of algebraic extension and roots of polynomials - this shares some common features with the recent work in \cite{baker2021descartes} - we show that every superfield has a (unique up to isomorphism) full algebraic extension to a superfield that is algebraically closed (Theorems \ref{uniq}, \ref{algclos}).

The next steps in this program are a development of Galois theory and Galois cohomology theory, envisaging application to other mathematical theories as abstract structures of quadratic forms and real algebraic geometry (\cite{ribeiro2016functorial},\cite{roberto2021quadratic},\cite{roberto2021ktheory}, \cite{roberto2021hauptsatz}).


\section{Multirings, Hyperfields}

\begin{defn}[Adapted from definition 1.1 in \cite{marshall2006real}]\label{defn:multimonoid}
 An \textbf{abelian or commutative multigroup} is a first-order structure  $(G,\cdot,r,1)$ where $G$ is a non-empty set, 
$r:G\rightarrow G$ is a function, $1$ is an element of $G$, $\cdot \subseteq G\times G\times G$ is a ternary relation 
(that will play the role of binary multioperation, we denote $d\in a\cdot b$ for $(a,b,d)\in\cdot$) such that for all 
$a,b,c,d\in G$:
 \begin{description}
 \item [M1 - ] If $c\in a\cdot b$ then $a\in c\cdot(r(b))\wedge b\in(r(a))\cdot c$. We write $a\cdot b^{-1}$ to simplify 
$a\cdot(r(b))$.
 \item [M2 - ] $b\in a\cdot1$ iff $a=b$.
 \item [M3 - ] If $\exists\,x(x\in a\cdot b\wedge t\in x\cdot c)$ then
$\exists\,y(y\in b\cdot c\wedge t\in a\cdot y)$.
 \item [M4 - ] $c\in a\cdot b$ iff $c\in b\cdot a$.
\end{description}
The structure $(G,\cdot,1)$ is a \textbf{commutative multimonoid (with unity)} if satisfy M3 and M4 and the condition 
$a\in1\cdot a$ for all $a\in G$.
\end{defn}

\begin{defn}[Adapted from Definition 2.1 in \cite{marshall2006real}]\label{defn:multiring}
 A multiring is a sextuple $(R,+,\cdot,-,0,1)$ where $R$ is a non-empty set, $+:R\times 
R\rightarrow\mathcal P(R)\setminus\{\emptyset\}$,
 $\cdot:R\times R\rightarrow R$
 and $-:R\rightarrow R$ are functions, $0$ and $1$ are elements of $R$ satisfying:
 \begin{enumerate}[i -]
  \item $(R,+,-,0)$ is a commutative multigroup;
  \item $(R,\cdot,1)$ is a commutative monoid;
  \item $a.0=0$ for all $a\in R$;
  \item If $c\in a+b$, then $c.d\in a.d+b.d$. Or equivalently, $(a+b).d\subseteq a.d+b.d$.
 \end{enumerate}

Note that if $a \in R$, then $0 = 0.a \in (1+ (-1)).a \subseteq 1.a + (-1).a$, thus $(-1). a = -a$.
 
 $R$ is said to be an hyperring if for $a,b,c \in R$, $a(b+c) = ab + ac$. 
 
 A multring (respectively, a hyperring) $R$ is said to be a multidomain (hyperdomain) if it hasn't zero divisors. A multring 
$R$ will be a 
multifield if every non-zero element of $R$ has 
multiplicative inverse; note that hyperfields and multifields coincide.
\end{defn}

\begin{ex}\label{ex:1.3}
$ $
 \begin{enumerate}[a -]
  \item Suppose that $(G,\cdot,1)$ is a group. Defining $a \ast b = \{a \cdot b\}$ and $r(g)=g^{-1}$, 
we have that $(G,\ast,r,1)$ is a multigroup. In this way, every ring, domain and field is a multiring, 
multidomain and multifield, respectively.
  
  \item Let $K=\{0,1\}$ with the usual product and the sum defined by relations $x+0=0+x=x$, $x\in K$ and 
$1+1=\{0,1\}$. This is a multifield  called Krasner's multifield \cite{jun2015algebraic}. 
  
  \item $Q_2=\{-1,0,1\}$ is ``signal'' multifield with the usual product (in $\mathbb Z$) and the multivalued sum defined by 
relations
  $$\begin{cases}
     0+x=x+0=x,\,\mbox{for every }x\in Q_2 \\
     1+1=1,\,(-1)+(-1)=-1 \\
     1+(-1)=(-1)+1=\{-1,0,1\}
    \end{cases}
  $$
  \end{enumerate}
\end{ex}

\begin{ex}[Tropical Hyperfield \cite{viro2010hyperfields}]\label{tropical-ex}

For a fixed totally ordered abelian group $(G, +, -, 0, \leq)$ we can construct a  {\em tropical multifield} $T_G = (G \cup \{\infty\}, \otimes, \odot, \ominus, 0, 1) $ where:
\begin{multicols}{2}
\begin{enumerate}[i -]
    \item $\forall g \in G, g < \infty$;
    \item $g \otimes h := g+h$;
    \item $0 := \infty$;
    \item $1 := 0$;
    \item $\forall g \in G$, $g^{-1} = -g$;
    \item if $g \neq h$, $g \oplus h = \{min\{g,h\}\}$;
    \item $g \oplus g = \{h \in G \cup \{\infty\} : g \leq h\}$; 
    \item $\ominus g = g$.
\end{enumerate}
\end{multicols}

\end{ex}

In the sequence, we provide examples that generalizes the previous ones.

\begin{ex}[H-multifield, Example 2.8 in \cite{ribeiro2016functorial}]\label{H-multi}
Let $p\ge1$ be a prime integer and $H_p:=\{0,1,...,p-1\} \subseteq \mathbb{N}$. Now, define the binary multioperation and 
operation in $H_p$ as 
follow:
\begin{align*}
 a+b&=
 \begin{cases}H_p\mbox{ if }a=b,\,a,b\ne0 \\ \{a,b\} \mbox{ if }a\ne b,\,a,b\ne0 \\ \{a\} \mbox{ if }b=0 \\ \{b\}\mbox{ if 
}a=0 \end{cases} \\
 a\cdot b&=k\mbox{ where }0\le k<p\mbox{ and }k\equiv ab\mbox{ mod p}.
\end{align*}
$(H_p,+,\cdot,-, 0,1)$ is a hyperfield such that for all $a\in H_p$, $-a=a$. In fact, these $H_p$ is a kind of generalization of $K$, in the sense that $H_2=K$.
\end{ex}

\begin{ex}[Kaleidoscope, Example 2.7 in \cite{ribeiro2016functorial}]\label{kaleid}
 Let $n\in\mathbb{N}$ and define 
 $$X_n=\{-n,...,0,...,n\} \subseteq \mathbb{Z}.$$ 
 We define the \textbf{$n$-kaleidoscope multiring} by 
$(X_n,+,\cdot,-, 0,1)$, where $- : X_n \to X_n$ is restriction of the  opposite map in$\mathbb{Z}$,  $+:X_n\times 
X_n\rightarrow\mathcal{P}(X_n)\setminus\{\emptyset\}$ is given by the rules:
 $$a+b=\begin{cases}
    \{a\},\,\mbox{ if }\,b\ne-a\mbox{ and }|b|\le|a| \\
    \{b\},\,\mbox{ if }\,b\ne-a\mbox{ and }|a|\le|b| \\
    \{-a,...,0,...,a\}\mbox{ if }b=-a
   \end{cases},$$
and $\cdot:X_n\times X_n\rightarrow X_n$ is is given by the rules:
 $$a\cdot b=\begin{cases}
    \mbox{sgn}(ab)\max\{|a|,|b|\}\mbox{ if }a,b\ne0 \\
    0\mbox{ if }a=0\mbox{ or }b=0
   \end{cases}.$$
With the above rules we have that $(X_n,+,\cdot, -, 0,1)$ is a multiring which is not a hyperring for $n\ge2$ because $$n(1-1)=b\cdot\{-1,0,1\}=\{-n,0,n\}$$
and $n-n=X_n$. Note that $X_0=\{0\}$ and $X_1=\{-1,0,1\} =  Q_2$. 
\end{ex}

\begin{ex}[Multigroup of a Linear Order, 3.4 of \cite{viro2010hyperfields}]
Let $(\Gamma,\cdot,1,\le)$ be an ordered abelian group. We have an associated hyperfield structure $(\Gamma\cup\{0\},+,-\cdot,0,1)$ with the rules $-a:=a$, $a\cdot0=0\cdot a:=0$ and
$$a+b:=\begin{cases}a \mbox{ if }a<b \\ b\mbox{ if }b<a\\
[0,a]\mbox{ if }a=b\end{cases}$$
Here we use the convention $0\le a$ for all $a\in\Gamma$. 
\end{ex}

 Now, we treat about morphisms:

\begin{defn}\label{defn:morphism}
 Let $A$ and $B$ multirings. A map $f:A\rightarrow B$ is a morphism if for all $a,b,c\in A$:
 \begin{multicols}{2}
 \begin{enumerate}[i -]
  \item $f(1)=1$ and $f(0)=0$;
  \item $f(-a)=-f(a)$;
  \item $f(ab)=f(a)f(b)$;
  \item $c\in a+b\Rightarrow f(c)\in f(a)+f(b)$.
 \end{enumerate}
 \end{multicols}
  A morphism $f$ is \textbf{a full morphism} if for all $a,b\in A$, $$f(a+b)=f(a)+f(b)\mbox{ and }f(a\cdot b)=f(a)\cdot f(b).$$
\end{defn}

\begin{ex}
$ $
\begin{enumerate}[i -]  
\item The prime ideals  of a commutative ring (its Zariski spectrum) are classified by equivalence classes of morphisms into  algebraically closed  fields,  but they can be {\em uniformly classified} by a multiring morphism into the Krasner multifield $K = \{0,1\}$.

\item The orderings of a commutative ring (its real spectrum) are classified by classes of equivalence of ring homomorphims into  real closed fields, but they can be {\em uniformly classified} by a multiring morphism into the signal multifield  $Q_2 = \{-1,0,1\}$. 

\item A Krull valuation on a commutative ring with group of values $(G, +, -, 0, \leq)$  is just a morphism into the multifield $T_G=G \cup\{\infty\}$.

\end{enumerate}

\end{ex}

\begin{lem}[Facts about full morphisms of multirings]\label{factstrong}
Let $f:A\rightarrow B$ be a full morphism of multirings. Then
\begin{enumerate}[a -]
    \item For all $a_1,...,a_n\in A$,
    $$f(a_1+...+a_n)=f(a_1)+...+f(a_n).$$
    \item For all $a_1,...,a_n,b_1,...,b_n\in A$,
    $$f[(a_1+b_1)(a_2+b_2)...(a_n+b_n)]=[f(a_1)+f(b_1)][f(a_2)+f(b_2))...(f(a_n)+f(b_n)].$$
    \item For all $c_1,...,c_n,d_1,...,d_n\in A$,
    $$f[c_1d_1+c_2d_2+...+c_nd_n]=f(c_1)f(d_1)+f(c_2)f(d_2)+...+f(c_n)f(d_n).$$
    
    \item For all $a_0,...,a_n,\alpha\in A$,
    $$f[a_0+a_1\alpha+...+a_n\alpha^n]=f(a_0)+f(a_1)f(\alpha)+...+f(a_n)f(\alpha)^n.$$
    
    \item Let $A_1,A_2,A_3$ be multirings with injective morphisms (embeddings) $i_{12}:A_1\rightarrow A_2$, $i_{13}:A_1\rightarrow A_3$ and $i_{23}:A_2\rightarrow A_3$. 
    $$\xymatrix@!=4pc{A_1\ar[r]^{i_{12}}\ar[dr]_{i_{13}} & A_2\ar[d]^{i_{23}} \\ & A_3}$$
    Suppose that $i_{13}=i_{23}\circ i_{12}$ is a full embedding. If $i_{23}$ is a full embedding then $i_{12}$ is a full embedding.
    
\end{enumerate}
\end{lem}
\begin{proof}
$ $
\begin{enumerate}[a -]
    \item By induction we only need to prove the case $n=3$. By the very definition of morphisms and multirings we get
    $$f(a_1+a_2+a_3)\subseteq f(a_1)+f(a_2)+f(a_3).$$
    To prove the another inclusion, let $d\in f(a_1)+f(a_2)+f(a_3)$. Since $f$ is full, $d\in f(a_1)+e$ for some $e\in f(a_2)+f(a_3)=f(a_2+a_3)$. Then $e=f(b)$ for some $b\in a_1+a_2$. Hence 
    $$d\in f(b)+f(a_3)=f(b+a_3)\subseteq f(a_1+a_2+a_3).$$
    Therefore $f(a_1)+f(a_2)+f(a_3)\subseteq f(a_1+a_2+a_3)$.
    
    \item By induction we only need to prove the case $n=2$. Let $y\in f[(a_1+b_1)(a_2+b_2)]$. Then $y=f(x)$ for some $x\in (a_1+b_1)(a_2+b_2)$. Moreover $x=x_1x_2$ for some $x_1\in a_1+b_1$ and $x_2\in a_2+b_2$. Then
    $f(x_1)\in f(a_1+b_1)=f(a_1)+f(b_1)$, $f(x_2)\in f(a_2+b_2)=f(a_2)+f(b_2)$ and
    $$y=f(x)=f(x_1x_2)=f(x_1)f(x_2)\in(f(a_1)+f(b_1))(f(a_2)+f(b_2)).$$
    Now let $z\in (f(a_1)+f(b_1))(f(a_2)+f(b_2))$. Then $z=z_1z_2$ for some
    $z_1\in f(a_1)+f(b_1)=f(a_1+b_1)$ and $z_2\in f(a_2)+f(b_2)=f(a_2+b_2)$. So there exist $x_1\in a_1+b_1$ and $x_2\in a_2+b_2$ with $z_1=f(x_1)$ and $z_2=f(x_2)$. Therefore
    $$z=z_1z_2=f(x_1)f(x_2)\in f[(a_1+b_1)(a_2+b_2)].$$
    
    \item Just apply item (b) changing the "variable" $a_i$ by $c_id_i$ and choosing $b_i=0$, $i=1,...,n$.
    
    \item Just apply item (c) with $c_i=a_i$ and $d_i=\alpha^i$, $i=0,...,n$.
    
    \item Suppose $i_{23}$ is a full embedding and that exist $d\in (i_{12}(a)+i_{12}(b))\setminus i_{12}(a+b)$ for some $a,b\in A_1$. Since $i_{23}$ is a full embedding, $i_{23}(d)\notin i_{23}(i_{12}(a+b))$. But
\begin{align*}
    i_{23}(d)&\in i_{23}(i_{12}(a)+i_{12}(b))=i_{23}\circ i_{12}(a)+i_{23}\circ i_{12}(b)=i_{13}(a)+i_{13}(b) \\
    &=i_{13}(a+b)=i_{23}\circ i_{12}(a+b)=i_{23}(i_{12}(a+b)),
\end{align*}
contradiction. Then $(i_{12}(a)+i_{12}(b))=i_{12}(a+b)$ and $i_{12}$ is a full embedding.
\end{enumerate}
\end{proof}

\section{Superrings, Superfields}

 The concept of suppering first appears in  (\cite{ameri2019superring}). There are many important advances and results in 
hyperring theory, and 
for instance, we recommend for example, the following papers: \cite{al2019some}, \cite{ameri2017multiplicative}, 
\cite{ameri2019superring}, 
\cite{ameri2020advanced}, \cite{massouros1985theory}, \cite{nakassis1988recent}, \cite{massouros1999homomorphic}, 
\cite{massouros2009join}.
 
\begin{defn}[Definition 5 in \textup{\cite{ameri2019superring}}]
 A superring is a structure $(S,+,\cdot, -, 0,1)$ such that:
 \begin{enumerate}[i -]
  \item $(S,+, -, 0)$ is a commutative multigroup.
  
  \item $(S,\cdot,1)$ is a commutative multimonoid. 
  
  \item $0$ is an absorbing element: $a\cdot0= \{0\} = 0 \cdot a$, for all $a\in S$.
  
  \item The weak/semi distributive law holds: 
   if $d\in c.(a+b)$ then $d\in (ca+cb)$,
  for all $a,b,c,d\in S$.  
  
  \item  The rule of signals holds:  $-(ab)=(-a)b=a(-b)$, for all $a,b\in S$.
 \end{enumerate}
 A superdomain is a non-trivial superring without zero-divisors in this new context, i.e. whenever
 $$0\in a\cdot b \mbox{ iff }a=0 \mbox{ or } b=0$$
 A quasi-superfield is a non-trivial superring such that every nonzero element is invertible in this new context\footnote{
 For a quasi-superfield $F$, we \textbf{are not imposing} that $(S\setminus\{0\},\cdot,1)$ will be a commutative multigroup, i.e, 
that if $d\in a\cdot b$ then $b^{-1}\in a\cdot d^{-1}$.}, i.e. 
whenever
 $$\mbox{ For all }a \neq 0 \mbox{ exists }b\mbox{ such that }1\in a\cdot b.$$
 A superfield is a quasi-superfield which is also a superdomain. A superring is full if for all $a,b,c,d\in S$, $d\in c\cdot(a+b)$ iff $d\in ca+cb$.
\end{defn}

\begin{ex}
Every multiring can be seen as a superring, in the very same fashion of \ref{ex:1.3}(a). Our main example of superring is the superring of multipolynomials $R[X]$ over a multiring $R$. The construction will be presented in short in Section \ref{secpol}. For more details, see \cite{roberto2021superrings}, \cite{ameri2019superring} or \cite{davvaz2016codes}.
\end{ex}

Now we treat about morphisms.

\begin{defn}
 Let $A$ and $B$ superrings. A map $f:A\rightarrow B$ is a morphism if for all $a,b,c\in A$:
 \begin{multicols}{2}
  \begin{enumerate}[i -]
  \item $f(0)=0$;
  \item $f(1)=1$;
  \item $f(-a)=-f(a)$;
  \item $c\in a+b\Rightarrow f(c)\in f(a)+f(b)$;
  \item $c\in a\cdot b\Rightarrow f(c)\in f(a)\cdot f(b)$.
 \end{enumerate} 
 \end{multicols}
 A morphism $f$ is \textbf{a full morphism} if for all $a,b\in A$, $$f(a+b)=f(a)+f(b)\mbox{ and }f(a\cdot b)=f(a)+f(b).$$
\end{defn}


From now on, we use the following conventions: Let $(R,+,\cdot, -, 0,1)$ be a superring, $p\in\mathbb N$ and consider a $p$-tuple $\vec a=(a_0,a_1, ..., a_{p-1})$. We define the finite sum by:
\begin{align*}
 x\in\sum_{i<0}a_i&\mbox{ iff }x=0, \\
 x\in\sum_{i<p}a_i&\mbox{ iff }x\in y+a_{p-1}\mbox{ for some }y\in\sum_{i<p-1}a_i, \text{if} \ p \geq 1.
\end{align*}
and the finite product by:
\begin{align*}
 x\in\prod_{i<0}a_i&\mbox{ iff }x=1, \\
 x\in\prod_{i<p}a_i&\mbox{ iff }x\in y\cdot a_{p-1}\mbox{ for some }y\in\prod_{i<p-1}a_i, \text{if} \ p \geq 1.
\end{align*}

Thus, if $(\vec{a}_0, \vec{a}_1,...,\vec{a}_{p-1})$ is a $p$-tuple of tuples $\vec{a}_i = (a_{i0}, a_{i1},..., a_{i{m_i}})$, 
then 
we have the finite sum of finite products:
\begin{align*}
 x\in\sum_{i<0}\prod_{j<{m_i}}a_{ij}&\mbox{ iff }x=0, \\
 x\in\sum_{i<p}\prod_{j<{m_i}}a_{ij}&\mbox{ iff }x\in y+z\mbox{ for some }y\in
 \sum_{i<{p-1}}\prod_{j<{m_i}}a_{ij}
\mbox{ 
and } z\in\prod_{j<m_{p-1}}a_{{p-1},j}, \text{if} \ p \geq 1.
\end{align*}

\begin{lem}[Basic Facts]\label{lembasic1}
    Let $A$ be a superring.
    \begin{enumerate}[a -]
        \item For all $n\in\mathbb N$ and all $a_1,...,a_n\in A$, the sum $a_1+...+a_n$ and product $a_1\cdot...\cdot  a_n$ does not depends on the order of the entries.
        \item If $A$ is a full superdomain, then $ax=ay$ for some $a\ne0$ imply $x=y$.
        \item If $A$ is full, then for all $d,a_1,...,a_n\in A$
        $$d(a_1+...+a_n)=da_1+...+da_n.$$
        \item Suppose $A$ is a full superdomain and let $a\in A\setminus\{0\}$. If $1\in (a\cdot b)\cap(a\cdot c)$ then $b=c$.
        \item (Newton's Binom Formula) For $n\ge1$ and $X\subseteq A$ denote
            $$nX:=\sum^n_{i=1}X.$$
            Then for $A,B\subseteq A$,
            $$(A+B)^n\subseteq\sum^n_{j=0}\binom{n}{j}A^jB^{n-j}.$$
    \end{enumerate}
\end{lem}
\begin{proof}
 $ $
 \begin{enumerate}[a -]
  \item It is an immediate consequence of associativity and induction. 
  \item Let $ax=ay$ for some $a\ne0$. Then $ax-ay=ay-ay$. Since $A$ is full, $a(x-y)=ay-ay$, and then, 
  $$0\in ay-ay=a(x-y).$$
  Moreover, $0\in az$ for some $z\in x-y$. Since $A$ is a superdomain and $a\ne0$, $z=0$. Then $0\in x-y$, which imply $x=y$.
  \item By induction, we only need to proof the case $n=2$. Let $a,b,c,d\in A$. We already know that $d(a+b+c)\subseteq da+db+dc$. Now consider $x\in da+db+dc$. Then $x\in e+dc$ for some $e\in da+db=d(a+b)$. Then $e\in de'$ with $e'\in a+b$ and $x\in e+dc\subseteq de'+dc=d(e'+c)$. Hence
  $$x\in d(e'+c)\subseteq d(a+b+c).$$
  \item Let $1\in (a\cdot b)\cap(a\cdot c)$. Then
  $$0\in1-1\subseteq(a\cdot b)-(a\cdot c)=a\cdot(b-c).$$
  Since $0\in a\cdot(b-c)$ and $a\ne0$ we have $0\in b-c$, which imply $b=c$.
  \item By induction is enough to prove the case $n=2$. We have
  \begin{align*}
    (A+B)^2&:=(A+B)(A+B)\subseteq A(A+B)+B(A+B)\subseteq A^2+AB+BA+B^2 \\
    &=A^2+AB+AB+B^2=A^2+2AB+B^2:=\sum^2_{j=0}\binom{n}{j}A^jB^{n-j}.
  \end{align*}
 \end{enumerate}
\end{proof}

\begin{lem}[Facts about full morphisms of superrings]\label{factstrong2}
Let $f:A\rightarrow B$ be a full morphism of superrings. Then
\begin{enumerate}[a -]
    \item For all $a_1,...,a_n\in A$,
    $$f(a_1+...+a_n)=f(a_1)+...+f(a_n).$$
    \item For all $a_1,...,a_n,b_1,...,b_n\in A$,
    $$f[(a_1+b_1)(a_2+b_2)...(a_n+b_n)]=(f(a_1)+f(b_1))(f(a_2)+f(b_2))...(f(a_n)+f(b_n)).$$
    \item For all $c_1,...,c_n,d_1,...,d_n\in A$,
    $$f(c_1d_1+c_2d_2+...+c_nd_n)=f(c_1)f(d_1)+f(c_2)f(d_2)+...+f(c_n)f(d_n).$$
    
    \item For all $a_0,...,a_n,\alpha\in A$,
    $$f(a_0+a_1\alpha+...+a_n\alpha^n)=f(a_0)+f(a_1)f(\alpha)+...+f(a_n)f(\alpha)^n.$$
    
    \item Let $A_1,A_2,A_3$ be superrings with injective morphisms (embeddings) $i_{12}:A_1\rightarrow A_2$, $i_{13}:A_1\rightarrow A_3$ and $i_{23}:A_2\rightarrow A_3$. 
    $$\xymatrix@!=4pc{A_1\ar[r]^{i_{12}}\ar[dr]_{i_{13}} & A_2\ar[d]^{i_{23}} \\ & A_3}$$
    Suppose that $i_{13}=i_{23}\circ i_{12}$ is a full embedding. If $i_{23}$ is a full embedding then $i_{12}$ is a full embedding.
\end{enumerate}
\end{lem}
\begin{proof}
Similar to Lemma \ref{factstrong}.
\end{proof}

\begin{defn}\label{char}
$ $
 \begin{enumerate}[i -]
  \item The \textbf{characteristic} of a superring is the smaller integer $n \geq 1$ such that
  $$0\in\sum_{i<n}1,$$
  otherwise the characteristic is zero.
  For full superdomains, this is equivalent to say that $n$ is the smaller integer such that
  $$\mbox{For all }a,\,0\in\sum_{i<n}a.$$
 
  \item An \textbf{ideal} of a superring $A$ is a non-empty subset $\mathfrak{a}$ of $A$ such that 
$\mathfrak{a}+\mathfrak{a}\subseteq\mathfrak{a}$ and $A\mathfrak{a}\subseteq\mathfrak{a}$. We denote
  $$\mathfrak I(A)=\{I\subseteq A:I\mbox{ is an ideal}\}.$$
  
  \item  Let $S$ be a subset of a superring $A$. We define the \textbf{ideal generated by} $S$ as 
  $$\langle S\rangle:=\bigcap\{\mathfrak{a}\subseteq A\mbox{ ideal}:S\subseteq\mathfrak{a}\}.$$
  If $S=\{a_1,...,a_n\}$, we easily check that
  $$\langle a_1,...,a_n\rangle=\sum Aa_1+...+\sum Aa_n,\,\mbox{where }\sum 
Aa=\bigcup\limits_{n\ge1}\{\underbrace{Aa+...+Aa}_{n\mbox{ times}}\}.$$ 
  Note that if $A$ is a full superring, then $\sum Aa=Aa$.
  
  \item An ideal $\mathfrak{p}$ of $A$ is said to be \textbf{prime} if $1\notin\mathfrak{p}$ and 
$ab\subseteq\mathfrak{p}\Rightarrow a\in\mathfrak{p}$ or $b\in\mathfrak{p}$. We denote 
  $$\mbox{Spec}(A)=\{\mathfrak{p}\subseteq A:\mathfrak{p}\mbox{ is a prime ideal}\}.$$
  
  \item An ideal $\mathfrak{p}$ of $A$ is said to be \textbf{strongly prime} if $1\notin\mathfrak{p}$ and 
$ab\cap\mathfrak{p}\ne\emptyset\Rightarrow a\in\mathfrak{p}$ or $b\in\mathfrak{p}$. We denote 
  $$\mbox{Spec}_s(A)=\{\mathfrak{p}\subseteq A:\mathfrak{p}\mbox{ is a strongly prime ideal}\}.$$
  Note that every strongly prime ideal is prime. 
  
  \item An ideal $\mathfrak{m}$ is maximal if it is proper and for all ideals $\mathfrak{a}$ with 
  $\mathfrak{m}\subseteq\mathfrak{a}\subseteq 
  A$ then $\mathfrak{a}=\mathfrak{m}$ or $\mathfrak{a}=A$.
  
  \item For an ideal $I\subseteq A$, we define operations in the quotient $A/I=\{x+I:x\in A\}=\{\overline x:x\in A\}$, by the 
rules
  \begin{align*}
      \overline x+\overline y&=\{\overline z:z\in x+y\}\\
      \overline x\cdot\overline y&=\{\overline z:z\in xy\}
  \end{align*}
    for all $\overline x,\overline y\in A/I$.
 \end{enumerate}
\end{defn}

\begin{rem}
$ $
 \begin{enumerate}[a -]
     \item If $A$ is a multiring, then every prime ideal is strongly prime. We do not know if this is the case for general superrings.
     
     \item If $A$ is a multiring, then every maximal ideal is prime (Proposition 1.7 of \cite{ribeiro2021anel}). For a general superring $A$, we do not know if a maximal ideal is prime.
     
     \item In his Ph.D Thesis \cite{ribeiro2021anel}, H. Ribeiro deals with elements \emph{weakly invertible} on a multiring $A$. This could be an anternative in dealing with the above questions.
 \end{enumerate}
\end{rem}

With all conventions and notations above, we obtain the following Lemma, which recover for superrings some properties holding for rings (and multirings).

\begin{lem}\label{lem1}
 Let $A$ be a superring and $I$ an ideal.
 \begin{enumerate}[i -]
  \item $I=A$ if and only if $1\in I$.
  \item $A/I$ is a superring. Moreover, if $A$ is full then $A/I$ is also full.
  \item $I$ is strongly prime if and only if $A/I$ is a superdomain.
 \end{enumerate}
 If $A$ is full, then
 \begin{enumerate}
     \item [iv -] $I=A$ if and only if $1\in I$, which occurs if and only if $A^*\cap I\ne\emptyset$ (in other words, if and only if 
$I$ contains an invertible element).
     \item [v -] $A$ is a superfield if and only if $\mathfrak I(A)=\{0,A\}$.
     \item [vi -] $I$ is maximal if and only if $A/I$ is a superfield.
 \end{enumerate}
\end{lem}



\section{Multipolynomials}\label{secpol}

Even if the rings-like multi-algebraic structure have been studied for more than 70 years, the developments of notions of polynomials in the ring-like multialgebraic structure seems to have a more significant development only from the last decade: for instance in \cite{jun2015algebraic} some notion of multi polynomials is introduced to obtain some applications to algebraic and tropical geometry, in \cite{ameri2019superring} a more detailed account of variants of concept of multipolynomials over hyperrings is applied to get a form of Hilbert's Basissatz.

Here we will stay close to the perspective in  \cite{ameri2019superring}: let $(R,+,-,\cdot,0,1)$ be a superring and set
$$R[X]:=\{(a_n)_{n\in\omega}\in R^\omega:\exists\,t\,\forall n(n\ge t\rightarrow a_n=0)\}.$$
Of course, we define the \textbf{degree} of $(a_n)_{n\in\omega}\ne\bm0$ to be the smallest $t$ such that $a_n=0$ for all $n>t$. 

Now define the binary multioperations $+,\cdot :  R[X]\times R[X] \to  \mathcal P^*(R[X])$, a unary operation $-:R[X]\rightarrow R[X]$ and elements $0,1\in R[X]$ by
\begin{align*}
 (c_n)_{n\in\omega}\in (a_n)_{n\in\omega}+(b_n)_{n\in\omega}&\mbox{ iff }\forall\,n(c_n\in a_n+b_n) \\
  (c_n)_{n\in\omega}\in((a_n)_{n\in\omega}\cdot (b_n)_{n\in\omega}&\mbox{ iff }\forall\,n
  (c_n\in a_0\cdot b_n+a_1\cdot b_{n-1}+...+a_n\cdot b_0) \\
  -(a_n)_{n\in\omega}&=(-a_n)_{n\in\omega} \\
  0&:=(0)_{n\in\omega} \\
  1&:=(1,0,...,0,...)
\end{align*}
For convenience, we denote elements of $R[X]$ by $\alpha=(a_n)_{n\in\omega}$. Beside this, we denote
\begin{align*}
 1&:=(1,0,0,...), \\
 X&:=(0,1,0,...), \\
 X^2&:=(0,0,1,0,...)
\end{align*}
etc. In this sense, our ``monomial'' $a_iX^i$ is denoted by $(0,...0,a_i,0,...)$, where $a_i$ is in the $i$-th position; in 
particular, we will denote ${\underline{b}} = (b,0,0,...)$ and we frequently identify $b \in R \leftrightsquigarrow 
{\underline{b}} \in R[X]$.

The properties stated  in the Lemma below immediately follows from the definitions involving $R[X]$:

\begin{lem}\label{lemperm}
 Let $R$ be a superring and $R[X]$ as above and $n,m\in \mathbb N$.
 \begin{enumerate}[a -]
  \item $\{X^{n+m}\}=X^n\cdot X^m$.
  \item For all $a\in R$, $\{aX^n\}= {\underline{a}}\cdot X^n$.
  \item Given $\alpha=(a_0,a_1,...,a_n,0,0,...)\in R[X]$, with with $\deg\alpha \leq n$ and $m\ge1$, we have
  $$\alpha X^m=(0,0,...,0,a_0,a_1,...,a_n,0,0,...)=a_0X^m+a_1X^{m+1}+...+a_nX^{m+n}.$$
  \item For $\alpha=(a_n)_{n\in\omega}\in R[X]$, with $\deg\alpha=t$, 
  $$\{\alpha\}=a_0\cdot1+a_1\cdot X+...+a_t\cdot X^t=a_0+X(a_1+a_2X+...+a_nX^{t-1}).$$
  \item $R[X]$ is a superdomain iff $R$ is a superdomain.
  \item $R[X]$ is a superring.
  \item The map $a \in R \mapsto {\underline{a}} = (a,0, \cdots,0, \cdots)$ defines a full embedding $R\rightarrowtail R[X]$.
  \item For an ordinary ring $R$ (identified with a strict suppering), the superring $R[X]$ is naturally isomorphic to (the superring associated to) the ordinary ring of polynomials in one variable over $R$.
   \end{enumerate}
\end{lem}

Lemma \ref{lemperm} allow us to deal with the superring $R[X]$ as usual. In other words, we can assume that for $\alpha\in R[x]$, there exists 
$a_0,a_1,...,a_n\in R$ such that $\alpha=a_0+a_1X+...+a_nX^n$, and then, we can work simply denoting $\alpha=f(X)$, as usual. For example, combining the definitions and all facts above we get
$$(x-a)(x-b)=x^2+(a-b)x+ab=\{x^2+dx+e:d\in a-b\mbox{ and }e\in ab\}.$$

\begin{rem} 
If $R$ is a full superdomain, does not hold in general that $R[X]$ is also a full superdomain. In fact, even if $R$ is a 
hyperfield, there are examples, e.g. $R = K, Q_2$, such that $R[X]$ is not a full superdomain (see \cite{ameri2019superring}).
\end{rem}

\begin{defn}
 The superring $R[X]$ will be called the \textbf{superring of polynomials} with one variable over $R$. The elements of $R[X]$ will be called 
 polynomials. We denote $R[X_1,...,X_n]:=(R[X_1,...,X_{n-1}])[X_n]$.
\end{defn}

\begin{lem}[Adapted from Theorem 5 of \cite{ameri2019superring}]\label{degreelemma}
Let $R$ be a superring and $f,g\in R[X]\setminus\{0\}$.
\begin{enumerate}[i -]
    \item If $t(X)\in f(X)+g(X)$ and $f\ne-g$ then $$\min\{\deg(f),\deg(g)\}\le\deg(t)\le\max\{\deg(f),\deg(g)\}.$$
    \item If $R$ is a superdomain and $t(X)\in f(X)g(X)$, then $\deg(t)=\deg(f)+\deg(g)$. In particular, if $f_1(X),f_2(X),...,f_n(X)\ne0$ and $t(X)\in f_1(X)f_2(X)...f_n(X)$, then
    $$\deg(t)=\deg(f_1)+\deg(f_2)+...+\deg(f_n).$$
    \item (Partial Factorization) Let $R$ be a superdomain, $\deg(f)=n$ and $f\in (X-a_1)(X-a_n)...(X-a_p)$. Then $p=n$.
\end{enumerate}
\end{lem}

Let $f(X)=a_0+...+a_nX^n$ and $g(X)=b_0+...+b_mX^m$ with $a_n,b_m\ne0$. We establish the following notation: for $k\in\mathbb N$ with $k\le\deg(f)$ we define $(f)_k:=a_k$ (the $k$-th coefficient of $f$).

\begin{proof}[Proof of Lemma \ref{degreelemma}]
For item (i), we have
$$f(X)+g(X)=(a_0+b_0)X+...+(a_n+b_m)X^m.$$
Since $f(X)\ne-g(X)$, $0\notin a_n+b_n$, establishing item (i).

Now, suppose without loss of generality that $m\ge n$ and in this case, write
$$f(X)=a_0+...+a_mX^m$$
with $a_k=0$ for $n<k\le m$. We have $(fg)_{m+n}\in a_nb_m$ and since $R$ is a superdomain, $(fg)_{m+n}\ne0$. This and induction proves item (ii).

For item (iii), let $g\in(X-a_1)(X-a_n)...(X-a_p)$. By item (ii) and induction, $\deg(g)=p$. Then $n=\deg(f)=p$.
\end{proof}

Despite the fact that $R[X]$ is not full in general, we have a powerful Lemma to get around this situation.

\begin{lem}\label{lemfator}
 Let $R$ be a superring and $f\in R[X]$ with $f(X)=a_nX^n+...+a_1X+a_0$. Then:
 \begin{enumerate}[i -]
     \item For all $b,c\in R$, $(b+cX)f(X)=bf(X)+cXf(X)$.
     \item For all $b,c\in R$ and all $p,q\in\omega$ with $p<q$, $$(bX^p+cX^q)f(X)=bX^pf(X)+cX^pf(X).$$
     \item For all $b,c,d\in R$ and all $p,q,r\in\omega$ with $p<q<r$, $$(bX^p+cX^q+dX^r)f(X)=bX^pf(X)+cX^pf(X)+dX^rf(X).$$
     \item For all $b_0,....,b_m\in R$,
     \begin{align*}
       (b_0+b_1X+b_2X^2+...+b_mX^m)f(X)=\\
     b_0f(X)+(b_1X+b_2X^2+...+b_mX^m)f(X).  
     \end{align*}
     \item For all $b_0,....,b_m\in R$,
     \begin{align*}
       (b_0+b_1X+b_2X^2+...+b_{m-1}X^{m-1}+b_mX^m)f(X)=\\
     (b_0+b_1X+b_2X^2+...+b_{m-1}X^{m-1})f(X)+b_mX^mf(X).  
     \end{align*}
     \item For all $b_0,....,b_m\in R$,
     \begin{align*}
       (b_0+b_1X+...+b_jX^j+b_{j+1}X^{j+1}+...+b_mX^m)f(X)=\\
     (b_0+b_1X+...+b_jX^j)f(X)+(b_{j+1}X^{j+1}+...+b_mX^m)f(X).  
     \end{align*}
     In particular, if $d\in R$, $g(X)\in R[X]$ and $r>\deg(g(X))$, then
     $$(g(X)+dX^r)f(X)=g(X)f(X)+dX^rf(X).$$
 \end{enumerate}
\end{lem}
\begin{proof}
 $ $
 \begin{enumerate}[i -]
     \item We can suppose without loss of generality that $b,c\ne0$. Here is convenient keep in mind that an element in $R[X]$ is a sequence of elements in $R$. Denote $b+cX=(b_n)_{n\in\omega}\in R[X]$ with $b_0=b$, $b_1=c$ and $b_n=0$ for all $n\ge2$. By definition, for an element $h(X)\in R[X]$, say
     $$h(X)=e_0+e_1X+...+e_{n+1}X^{n+1}=(e_n)_{n\in\omega}\in R[X],$$
     we have
     $$h(X)\in(b+cX)f(X)\mbox{ iff }e_p\in\sum^p_{j=0}a_jb_{p-j},\,p\in\omega.$$
     Since $a_j=0$ for all $j>n$ and $b_j=0$ for all $j\ge2$, we have $e_p=0$ for all $p>n+1$. Moreover, by the same reason we have that $e_0\in a_0b_0$, $e_{n+1}\in a_nb_1$ and for $0< p< n+1$, that
     $$e_p\in\sum^p_{j=0}a_jb_{p-j}=a_pb_0+a_{p-1}b_1.$$
     Summarizing, we conclude that 
     \begin{align*}
         h(X)\in(b+cX)f(X)&\mbox{ iff }e_0\in a_0b_0,\,e_{n+1}\in a_nb_1\mbox{ and }
         e_p\in a_pb_0+a_{p-1}b_1\mbox{ for }0< p< n+1.\tag{*}
     \end{align*}
     On the other hand, we have that
     \begin{align*}
         bf(X)+cXf(X)&=
         b[a_nX^n+...+a_1X+a_0]+cX[a_nX^n+...+a_1X+a_0] \\
         &=(a_nbX^n+...+a_1bX+a_0b)+(a_ncX^{n+1}+...+a_1cX^2+a_0X) \\
         &=a_ncX^{n+1}+(a_nb+a_{n-1}c)X^n+...+(a_2b+a_1c)X^2+(a_1b+a_0c)X+a_0b\tag{**}.
     \end{align*}
     Joining ($*$) and ($**$) we conclude that
     $$h(X)\in(b+cX)f(X)\mbox{ iff }h(X)\in bf(X)+cXf(X).$$
     
     \item Just use the same reasoning of item (i).
     
     \item Using distributivity, item (i) and (ii) we conclude that
     \begin{align*}
         \tag{***}(bX^p+cX^q+dX^r)f(X)\subseteq bX^pf(X)+cX^pf(X)+dX^rf(X)
     =(bX^p+cX^p)f(X)+dX^rf(X).
     \end{align*}
     Now with ($***$) on hand, just proceed with the same reasoning of ($*$) and ($**$) to obtain the desired. 
     
     \item This is an immediate consequence of item (iii) and a convenient induction.
     
     \item This is an immediate consequence of item (iii) and a convenient induction.
     
     \item This is just the combination of previous items.
 \end{enumerate}
\end{proof}

\begin{teo}[Euclid's Division  Algorithm (3.4 in \cite{davvaz2016codes})]\label{euclid}
 Let $K$ be a superfield. Given polynomials $f(X),g(X)\in K[X]$ with $g(X)\ne0$, there exists $q(X),r(X)\in 
K[X]$ such that $f(X)\in q(X)g(X)+r(X)$, with $\deg r(X)<\deg g(X)$ or $r(X)=0$.
\end{teo}

\begin{proof}
This is a generalized version of Theorem 3.4 in \cite{davvaz2016codes}, which states Euclid's Algorithm for hyperfields.
Write
\begin{align*}
 f(X)&=a_nX^n+ \cdots +a_1X+a_0\\
 g(X)&=b_mX^m+ \cdots +b_1X+b_0
\end{align*}
with $a_n,b_m\ne0$ and let $b_m^{-1} \in K$ be an element satisfying $1 \in b_m\cdot b_m^{-1}$.

We proceed by induction on $n$. Note that if $m\ge n$, then is sufficient take $q(X)=0$ 
and $r(X)=f(X)$, so we can suppose $m\le n$. If $m=n=0$, then $f(X)=a_0$ and $g(X)=b_0$ are both non zero constants, so is sufficient take $q(X)\in a_0\cdot b_0^{-1}$ and $r(X)=0$.

Now, suppose $n\ge1$. Then, since $0\in a-a$, there exist some $t(X)\in f(X)-a_nb_m^{-1}X^{n-m}g(X)$ with $\deg t(X)<n$. So, by induction hypothesis,
$$t(X)\in q(X)g(X)+r(X)\mbox{ for some }q(X),r(X)\in R[X]\mbox{ with }\deg r(X)<\deg g(X)\mbox{ or }r(X)=0.$$
Therefore, $\deg t(X) = \deg q(X) + m$ and since $f(X)\in t(X)+a_nb_m^{-1}X^{n-m}g(X)$, we have
\begin{align*}
    f(X)&\in t(X)+a_nb_m^{-1}X^{n-m}g(X) \\
    &\subseteq q(X)g(X)+a_nb_m^{-1}X^{n-m}g(X)+r(X).
\end{align*}

But since $\deg q(X) = \deg t(X) - m < n - m$, we have (see Lemma \ref{lemfator} (vi)) that
$$[q(X)+a_nb_m^{-1}X^{n-m}]g(X) = q(X)g(X)+a_nb_m^{-1}X^{n-m}g(X).$$
So there exist some $q'(X)\in q(X)+a_nb_m^{-1}X^{n-m}$ with $f(X)\in q'(X)g(X)+r(X)$ and $\deg r(X)<\deg g(X)$ or $r(X)=0$, completing the proof.
\end{proof}

\begin{rem}
$ $
 \begin{enumerate}[i -]
  \item Note that the polynomials $q$ and $r$ of Theorem \ref{euclid} are not unique in general: if $f\in gq+r$, then 
$f\in g(q+1-1)+r$ and $f\in gq+(r+1-1)$, then, if $\{0\}\ne1-1$, we have many $q$'s and $r$'s.

 \item However, if $R$ is a ring, then Theorem \ref{euclid} provide the usual Euclid Algorithm, with the uniqueness of the quotient and remainder.
\end{enumerate}
\end{rem}

\begin{teo}[Adapted from Theorem 6 of \cite{ameri2019superring}]\label{teoPID}
 Let $F$ be a full superfield. Then $F[X]$ is a principal ideal superdomain.
\end{teo}
\begin{proof}
Let $I$ be a ideal of $F[X]$. If $I=0$ then $I=\langle0\rangle$ and if there is some $a\in F\setminus\{0\}$ with $a\in I$, then $I=F[X]=\langle 1\rangle$ (because $F$ is full).

Now let $p(X)\in I$ be a polynomial with minimal degree $m\ge1$. Let $f(X)\in I$ be another polynomial. By Euclid's Algorithm, there exists $q(X),r(X)\in F[X]$ with $f(X)\in p(X)q(X)+r(X)$ and $r(X)=0$ or $\deg(r)<\deg(p)=m$. Since $f,p\in I$ and $r(X)\in f(X)-p(X)q(X)$, we have $r\in I$. Note that by the minimality of $m$, all nonzero polynomial in $f(X)-p(X)q(X)$ has degree at least $m$. If $r\ne0$ then 
$$\min\{\deg f,\deg(p)+\deg(q)\}\le\deg r\le\max\{\deg f,\deg(p)+\deg(q)\}.$$
In particular $\deg(r)\ge m$ (because $\deg(f)\le m$), contradicting $deg(r)<m$. Hence $r=0$ and $I=\langle p\rangle$. In particular, $I=F[X]\cdot p(X)$.
\end{proof}

\section{Evaluation and Roots}

Let $R, S$ be  superrings and $h : R \to S$ be a morphism. Then $h$ extends naturally to a morphism in the superrings multipolynomials $h^X : 
R[X] \to S[X]$:
$$(a_n)_{n \in \mathbb{N}} \in R[X] \ \mapsto \ (h(a_n))_{n \in \mathbb{N}} \in S[X]$$

Now let $s \in S$. We define the $h$-\textbf{evaluation} of $s$ at $f(X)\in R[X]$ with $f(X)=a_0+a_1X+...+a_nX^n$ by
$$f^h(s)=ev^h(s,f):=\{s'\in S : s'\in h(a_0)+ h(a_1).s+h(a_2).s^2+...+h(a_n).s^n\}.$$
We define the $h$-\textbf{evaluation} for a subset $I\subseteq S$ by
$$f^h(I)=\bigcup_{s\in I}f^h(s).$$

In particular if $S\supseteq R$ are superrings and $\alpha\in S$, we have the \textbf{evaluation} of $\alpha$ at  $f(X)\in R[X]$ by
$$f(\alpha,S)=ev(\alpha,f,S)=\{b\in S: b \in a_0+a_1\alpha+a_2\alpha^2+...+a_n\alpha^n\}\subseteq S.$$
Note that the evaluation \textbf{depends} on the choice of $S$. When $S=R$ we just denote $f(\alpha,R)$ by $f(\alpha)$. 

A \textbf{root} of $f$ in $S$ is an element $\alpha\in S$ such that $0\in ev(\alpha,f,S)$. In this case we say that $\alpha$ is \textbf{$S$-algebraic} over $R$. An \textbf{effective root} of $f$ in $S$ is an element $\alpha\in S$ such that $f\in(X-\alpha)\cdot g(X)$ for some $g(X)\in R[X]$. A superring $R$ is \textbf{algebraically closed} if every non constant polynomial in $R[X]$ has a root in $R$. 

Observe that, if $F$ is a field, the evaluation of $F[X]$ as a ring coincide with the usual evaluation, and, of course, root and effective roots are the same thing. Therefore, if $F$ is algebraically closed as hyperfield and superfield, then will be algebraically closed in the usual sense. 

\begin{rem}  \label{unexpected-rem}
The  expansion of the above field-theoretical concepts to the multialgebraic theory  of superfields (hyperfields, in particular)  brings new phenomena:

\begin{enumerate}[i-]

\item (Polynomials can have infinite roots):
Let $F$ be a infinite pre-special hyperfield  (\cite{ribeiro2016functorial}). Then $F$ has characteristic $0$, $a^2=1$ for all $a\ne0$ so the polynomial $f(X)=X^2-1$ has infinite roots (i.e, $0\in ev(f,\alpha)$ for all $\alpha\in\dot F$).

\item (Finite hyperfields can be algebraically closed). The hyperfield $K=\{0,1\}$ is algebraically closed. In fact, if $p(X)=a_0+a_1X+a_2X^2+...+a_nX^n\in K[X]$, with $a_n\ne0$, then $0\in p(0)$ (if $a_0=0$) or $p(1)=K$, since $1+1=\{0,1\}$.

\end{enumerate}
\end{rem}

We have good results concerning irreducibly (see for instance, Theorem \ref{lemquadext} below). These results are the key to the development of superfields extensions, which leads us to some kind of algebraic closure.

\begin{defn}[Irreducibility]
Let $R$ be a superfield and $f,d\in R[X]$. We say that $d$ divides $f$ if and only if $f\in\langle d\rangle$, and denote $d|f$. 
We say that $f$ is \textbf{irreducible} if $\deg f\ge1$ and $u|f$ for some $u\in R[X]$ (i.e, $f\in\langle u\rangle$), then 
$\langle f\rangle=\langle u\rangle$. 
\end{defn}

\begin{teo}\label{lemquadext2}
Let $F$ be a full superfield and $p(X)\in F[X]$ be an irreducible polynomial. Then $\langle p(X)\rangle$ is a maximal ideal.
\end{teo}
\begin{proof}
   Let $p(X)$ be irreducible and $I\subseteq F[X]$ an ideal with $\langle p(X)\rangle\subseteq I$. By Theorem \ref{teoPID}, 
   $$I=\langle f(X)\rangle=F[X]\cdot f(X)$$
   for some $f(X)\in F[X]$. Since $p(X)\in I=\langle f(X)\rangle$, then $p(X)=f(X)g(X)$ for some $g(X)\in F[X]$. Since $p(X)$ is irreducible, either $f(X)$ or $g(X)$ is a constant polynomial. If $f(X)$ is constant, then $I=F[X]$, and if $g(X)$ is constant, $I=\langle p(X)\rangle$, which proves that $\langle p(X)\rangle$ is maximal.
\end{proof}

If $F$ is not full, we cannot prove that $\langle p(X)\rangle$ is a maximal ideal. But we still have that $F[X]/\langle p\rangle$ is a superfield.
\begin{teo}\label{lemquadext}
 Let $F$ be a superfield and $p\in F[X]$ be an irreducible polynomial. Then $F[X]/\langle p\rangle$ is a superfield. In particular, $\langle p\rangle$ is a strongly prime.
\end{teo}
\begin{proof}
Let $p(X)=d_0+a_1X+...+a_{n+1}X^{n+1}$. Note that
   \begin{align}\label{tret}
     F(p(X)):=F[x]/\langle p\rangle&=\{[a_0+a_1X+...+a_nX^{n}]:a_0,...,a_n\in F\}
     \nonumber \\
     &=\{[f(X)]:f(X)=a_0+a_1X+...+a_rX^{r}\mbox{ with }a_0,...,a_r\in F,\,r\le n\}.
   \end{align}
   Let $f(X)=a_0+a_1X+...+a_rX^{r}$ and $g(X)=b_0+b_1X+...+b_sX^{s}$ with  and suppose 
   $$[0]\in[f(X)][g(X)].$$
   There exist
   $$h(X)\in (f(X)g(X))\cap\langle p(X)\rangle.$$ 
   Since $F$ is a superdomain, every nonzero polynomial in $\langle p\rangle$ has degree at least $n+1=\deg(p)$. Now get a nonzero element in $[t(x)]\in[f(X)][g(X)]$. Using Equation \ref{tret} we have $t(X)\in f(X)g(X)$ with $\deg(t)\le n$. Then $h(X)=0$ and $0\in f(X)g(X)$, which imply $f(X)=0$ or $g(X)=0$ (because $F[X]$ is a superdomain). Then $[F(X)]=0$ or $[g(X)]=[0]$, proving that $F[p(X)]$ is a superdomain (and then, $\langle p(X)\rangle$ is strongly prime).

Now we prove that $F[p(X)]$ is a superfield, i.e, that for all nonzero $[f(X)]\in F[p(X)]$, there exist a nonzero $[g(X)]\in F[p(X)]$ with $[1]\in [f(X)][g(X)]$. We proceed by induction on $n=\deg(f(X))$.

If $n=0$, then $f(X)=a$ for some $a\in\dot F$, and there exist $a^{-1}\in\dot F$\footnote{Of course, not necessarily unique.} with $1\in a\cdot a^{-1}$, and then $[1]\in[f(X)][a^{-1}]$. If $n=1$, then $f(X)=aX+b$, $a,b\in F$ ($a\ne0$). By Euclid's Algorithm, there exists $q(X),r(X)$ with $p(X)\in f(X)q(X)+r(X)$ with $r(X)=0$ or $\deg(r(X))<\deg(f(X))$. Since $p(X)$ is irreducible, $r(X)\ne0$ and $r(X)=d\in\dot F$. Moreover for some $d^{-1}\in\dot F$ with $1\in d\cdot d^{-1}$ we have
\begin{align*}
    p(X)&\in f(X)q(X)+d\Rightarrow[0]\in [f(X)][q(X)]+[d]\Rightarrow
-[d]\in[f(X)][q(X)]\\
&\Rightarrow[dd^{-1}]\subseteq[f(X)](-[d^{-1}][q(X)])\Rightarrow[1]\in[f(X)](-[d^{-1}][q(X)]),
\end{align*}
and then, there exist $[t(X)]\in-[d^{-1}][q(X)]$ with $[1]\in[f(X)][t(X)]$.

Now, suppose by induction that all polynomial of degree at most $n$ has an inverse and let $f(X)\in F[X]$ with $\deg(f(X))=n+1$. By Euclid's Algorithm, there exists $q(X),r(X)$ with $p(X)\in f(X)q(X)+r(X)$ with $r(X)=0$ or $\deg(r(X))<\deg(f(X))$ and since $p(X)$ is irreducible, we have $r(X)\ne0$. By induction hypothesis, there exist $g(X)\in F[X]$ with $[1]\in[r(X)][g(X)]$. Then
\begin{align*}
  p(X)\in f(X)q(X)+r(X)&\Rightarrow[0]\in [f(X)][q(X)]+[r(X)] \\
  &\Rightarrow[r(X)]\in -[f(X)][q(X)] \\
  &\Rightarrow[r(X)][g(X)]\subseteq -[f(X)][q(X)][g(X)] \\
  &\Rightarrow1\in[r(X)][g(X)]\subseteq[f(X)](-[q(X)][g(X)]),
\end{align*}
then there exist $[t(X)]\in-[q(X)][g(X)]$ with $[1]\in[f(X)][t(X)]$, completing the proof.
\end{proof}

Using Theorem \ref{lemquadext}, we obtain an algorithm to determine the invertible elements in $F[p(X)]$ particularly useful in the field case:

\begin{cor}\label{lemquadext3}
Let $F$ be a field and $p(X)\in F[X]$ be an irreducible polynomial. If $f(X)\ne0$ and $p(X)=f(X)q(X)+r(X)$ with $r(X)\ne0$, then
$$[f(X)]^{-1}=-[q(X)][r(X)]^{-1}\in F[p(X)].$$
\end{cor}

\begin{defn}
Let $F$ be a superfield and $p(X)\in F[X]$ be an irreducible polynomial. We denote $F(p):=F(p(X))=F[X]/\langle p(X)\rangle$.
\end{defn}

\begin{lem}\label{lemfator2}
 Let $F$ be a superfield and $p(X)\in F[X]$ be an irreducible polynomial. Denote $\overline X=\lambda$ and let $f\in F(p)$ with $f=\overline{a_n}\lambda^n+...+\overline{a_1}\lambda+\overline a_0$. Then:
 \begin{enumerate}[i -]
     \item For all $b,c\in F$, $(\overline b+\overline c\lambda)f=\overline bf+\overline c\lambda f$.
     \item For all $b_0,....,b_m\in F$,
     \begin{align*}
       (\overline{b_0}+\overline{b_1}\lambda+...+\overline{b_j}\lambda^j+\overline{b_{j+1}}\lambda^{j+1}+...+\overline{b_m}\lambda^m)f=\\
     (\overline{b_0}+\overline{b_1}\lambda+...+\overline{b_j}\lambda^j)f+(\overline{b_{j+1}}\lambda^{j+1}+...+\overline{b_m}\lambda^m)f.  
     \end{align*}
     In particular, if $d\in F$, $g\in F(p)$ with $g=\overline{b_0}+\overline{b_1}\lambda+\overline{b_2}\lambda^2+...+\overline{b_m}\lambda^m$ and $r>m$, then
     $$(g+\overline d\lambda^r)f=gf+\overline d\lambda^rf.$$
 \end{enumerate}
\end{lem}
\begin{proof}
 Similar to Lemma \ref{lemfator}.
\end{proof}

\begin{teo}\label{root1}
 Let $F$ be a superfield and $p(X)\in F[X]$ be a polynomial  of degree greater or equal to 1. Then there exist superfield $L$ such that $F\subseteq L$, $F$ is a sub superfield of $L$ (i.e, the inclusion $F\hookrightarrow L$ is a full morphism) and $p(X)$ has a root.
\end{teo}
\begin{proof}
 It is enough to show the result for $p(X)$ irreducible. In this case, the ideal $\langle p(X)\rangle\subseteq F[X]$ is maximal and 
$K'=F[X]/\langle 
p(X)\rangle$ is a superfield. If we consider the canonical injection $\iota:F\rightarrow F[X]/\langle p\rangle$ given by 
$a\mapsto\overline a$, we have a full morphism (basically because $F\hookrightarrow F[X]$ is full). Putting $F'=\iota(F)$ we have that $F\cong F'$, $F'\hookrightarrow L$ is a full morphism and the polynomial $p^\iota$ (given by the application of $\iota$ in each coefficient) has a root $\overline x$. 

Next, let $K=F\cup X$ for some $X$ of cardinality $K'\setminus F'$. We construct a bijection $\varphi:K\rightarrow K'$ which 
restrict ton $F$ is 
equal to $\iota$. This bijection transport the structure of superfield for $K$ (in the obvious way), in order to get an 
extension $K|F$ such that 
$f$ has a root $\varphi^{-1}(\overline x)$.
\end{proof}

\begin{cor}
Let $F$ be a superfield and $f\in F[X]$ be a polynomial with $n=\deg(f)\ge1$. Then there exist a superfield $L$ such that $F\subseteq L$ and $f$ has at least $n$ roots.
\end{cor}

\begin{cor}
Let $F$ be a superfield and $f_1,...,f_n\in F[X]$ be polynomials with $1\le \deg(f_j)=r_j$, $j=1,...,n$. Then there exist a superfield $L$ such that $F\subseteq L$ andt each $f_j$ has at least $r_j$ roots.
\end{cor}

\section{Extensions}

We have some possibilities to consider in order to define the notion of extension for superfields:

 \begin{defn}[Extensions]\label{extension}
Let $F$ and $K$ be superfields.
\begin{enumerate}[i -]
    \item We say that $K$ is a \textbf{proto superfield extension (or just a proto extension)} of $F$, notation $K|_pF$, if $F\subseteq K$.
    \item We say that $K$ is a \textbf{superfield extension (or just an extension)} of $F$, notation $K|F$ if $F\subseteq K$ and the inclusion map $F\hookrightarrow K$ is a superfield morphism.
    \item We say that $K$ is a \textbf{full superfield extension (or just a full extension)} of $F$, notation $K|_fF$ if $F\subseteq K$ and the inclusion map $F\hookrightarrow K$ is a full superfield morphism.
\end{enumerate}
\end{defn}

\begin{ex}
$ $
\begin{enumerate}[i -]
    \item Of course, all full extension is an extension and all extension is a proto extension.
    
    \item We have $K\subseteq Q_2$ but the inclusion map $K\hookrightarrow Q_2$ is not a morphism. Then we have a proto extension $Q_2|_pK$ that is not an extension.
    
    \item For $p,q$ prime integers with $q\ge p$ we have an inclusion morphism $H_p\hookrightarrow H_q$, but this morphism is not full. Then we have an extension $H_q|H_p$ that is not a full extension.
    
    \item Let $F$ be a superfield, $p\in F[X]$ an irreducible polynomial and $F(p)=F[X]/\langle p\rangle$ be the superfield built in Theorem \ref{root1}. Then we have a full morphism $F\hookrightarrow F(p)$ so we have a full extension $F(p)|_fF$.
    
    \item Let $F,K$ be fields such that $F\subseteq K$. Then the field extension $K|F$ satisfy all conditions in Definition \ref{extension}.
\end{enumerate}
\end{ex}

The result below justify a deeper look at full superfield extensions.

\begin{teo}\label{unicityext}
Let $K_1|_fF$ and $K_2|_fF$ be full superfield extensions and suppose that $\gamma\in K_1\cap K_2$. Then
$$F[\gamma,K_1]=F[\gamma,K_2].$$
\end{teo}
\begin{proof}
Suppose first that $K_2|_fK_1$ is a full extension. Then for all $f\in F[X]$, $ev(f,K_1)=ev(f,K_2)$, so $F[\gamma,K_1]=F[\gamma,K_2]$.

Now, for the general case just note that $K_1|_f(K_1\cap K_2)$ and $K_2|_f(K_1\cap K_2)$. Then
$$F[\gamma,K_1]=F[\gamma,K_1\cap K_2]=F[\gamma,K_2].$$
\end{proof}

\begin{defn}[Algebraic Extensions]
We say that a proto extension $K|_pF$ is \textbf{algebraic} if all element $\alpha\in K$ is $K$-algebraic over $F$. We denote the same for extensions and full extensions.
\end{defn}

\begin{defn}[Linear Independency, Basis, Degree]
Let $K|_pF$ be a proto extension and $I\subseteq K$. We say that $I$ is \textbf{$F$-linearly independent} if for all 
distinct $\lambda_1,...,\lambda_n\in I$, $n\in\mathbb N$, the following hold:
$$\mbox{If }0\in a_1\lambda_1+...+a_n\lambda_n \mbox{ then }a_1=...=a_n=0$$
and $I$ is \textbf{$F$-linearly dependent} if it is not $F$-linearly independent. We say that $I$ is a \textbf{$F$-basis} of 
$K$ if $I$ is linearly independent and $K$ is \textbf{generated by $I$}, i.e,
$$K=\bigcup_{n\ge0}\left\lbrace\sum^n_{i=0}a_i\lambda_i:a_i\in F,\,\lambda_i\in I\right\rbrace.$$
In this case, we write $K=F[I]$. We define the \textbf{degree} of $K|_pF$, notation $[K:F]$, by 
the following
$$[K:F]:=\infty\mbox{ or }[K:F]:=\max\{n:\mbox{the set }\{1,\lambda,\lambda^2,...,\lambda^n\}\mbox{ is linearly independent for all }\lambda\in K\}.$$
\end{defn}

\begin{rem}\label{rem1}
There are these immediate consequences of the above definitions:
\begin{enumerate}[a -]
 \item If $I\subseteq K$ is linearly independent and $J\subseteq I$ then $J$ is also linearly 
independent.
 \item An element $\alpha\in K$ is $F$-algebraic if and only if $\{\alpha^k:k\in\mathbb N\}$ is $F$-linearly dependent.
 \item If $[K:F]<\infty$ then all $\alpha\in K$ is $F$-algebraic.
 \item Let $F$ be a superfield and $p\in F[X]$ an irreducible polynomial, say $p(X)=a_0+a_1X+...+a_nX^{n-1}+X^n$. Then $\{\overline1,\overline X,...,\overline X^{n-1}\}$ is a $F$-basis of $F(p)$.
\end{enumerate}
\end{rem}

Now, let $K|_pF$ be a proto extension and $\gamma\in K$ algebraic. Then there exist an irreducible polynomial $f(X)$ such that $0\in f(\gamma,K)$. Let $\mbox{Irr}_F(\gamma,K)$ be the minimum degree irreducible polynomial $f(X)$ 
such that $0\in f(\gamma,K)$. Let $F[\gamma,K]\subseteq K$ be the set
$$F[\gamma,K]:=\bigcup_{f\in F[X]}ev(f,\gamma,K)\subseteq K,$$
and $I_{\gamma,K}\subseteq F[\gamma,K]$ the set
$$I_{\gamma,K}:=\bigcup_{f\in \langle\mbox{Irr}_F(\gamma,K)\rangle}ev(f,\gamma,K)\subseteq K.$$
Note that for all $g\in F[X]$ and all $a_0,...,a_n\in F$, applying the ``Newton's binom formula'' we get
$$ev(g,(a_0+a_1\gamma+a_2\gamma^2+...+a_{n-1}\gamma^{n-1}+a_n\gamma^n),K)\subseteq F[\gamma,K].$$

\begin{rem}
 $ $
 \begin{enumerate}[i -]
    \item If $K|F$ is a field extension then our $F[\gamma,K]$ coincide with the usual simple extension $F(\gamma)$.
    
    \item If $K|F$ is a superfield extension and $\gamma\in K$, then $F[\gamma,K]$ \textbf{depends on the choice of $K$}. For example, consider $H_3|H_1$ and $H_5|H_1$ and the element $2\in H_3$ (and of course, in $H_5$). Then
    \begin{align*}
      H_2[2,H_3]&=\bigcup_{f\in H_2[X]}ev(f,\gamma,H_3)=H_3, \\
      H_2[2,H_5]&=\bigcup_{f\in H_2[X]}ev(f,\gamma,H_5)=H_5,
    \end{align*}
    and then, $H_2[2,H_3]\ne H_2[2,H_5]$.
  
  \item For a proto extension $K|_pF$ the set $F[\gamma,K]$ may not be a superfield! Let $F=H_2$, $K=\mathbb R$ and $\gamma=2$. Then
  $$H_2[2,\mathbb R]=2\mathbb Z$$
  which is not a superfield.
 \end{enumerate}
\end{rem}

At this point, our goal is to obtain an appropriate notion for simple extensions of superfields. In other words, given a full extension $K|_fF$ and $\alpha\in K$ algebraic, it is highly desirable to obtain a superfield $F(\alpha)$ that:
\begin{enumerate}
    \item $F\cup\{\alpha\}\subseteq F(\alpha)$;
    \item $F(\alpha)$ is the minimal superfield (with respect to inclusion) satisfying (1);
    \item $F(\alpha)$ is "computable" in some way (or saying it in a more realistic manner, we want that $F(\alpha)\cong F(p)$ with $p(X)=\mbox{Irr}_F(\alpha)$)\footnote{As we will see later, simple calculations with superfield are highly demanding...}.
\end{enumerate}

For general superfields there are some obstacles to achieve this goal. The very first one is the fact that $R[X]$ is not full in general. However, we have an interesting property valid for all $a,b\in R[X]$:
$$a(1+X)=a+aX\mbox{ and }(a+b)X=aX+bX.$$
This property is the inspiration for the following definition.


\begin{defn}\label{almostfull}
Let $K|_pF$ be a proto superfield extension and $\gamma\in K$. Suppose that $K$ is $F$-generated by $\{1,\gamma^2,...,\gamma^n\}$. We say that $K$ is \textbf{$F$-almost full relative to $\gamma$ (or just almost full)} if for all $a,b,c\in F$, and all $p,q,r\in\mathbb N$ distinct
$$(a\gamma^p+b\gamma^q+c\gamma^r)\gamma=a\gamma^{p+1}+b\gamma^{q+1}+c\gamma^{r+1}.$$
\end{defn}

Here are some immediate consequences of Definition \ref{almostfull}:

\begin{lem}\label{lemfator3}
 Let $K|_fF$ be a full extension $F$-almost full relative to $\gamma$ and let $A=a_0+a_1\gamma+a_2^2+...+a_n\gamma^n$. Then:
 \begin{enumerate}[i -]
     \item For all $b,c\in F$, $(b+c\gamma)A=bA+c\gamma A$.
     \item For all $b_0,....,b_m\in F$,
     \begin{align*}
      (b_0+b_1\gamma+...+b_j\gamma^j+b_{j+1}\gamma^{j+1}+...+b_m\gamma^m)A=\\
     (b_0+b_1\gamma+...+b_j\gamma^j)A+(b_{j+1}\gamma^{j+1}+...+b_m\gamma^m)A.
     \end{align*}
     In particular, if $d\in F$, $B\subseteq K$ with $B=b_0+b_1\gamma+b_2\gamma^2+...+b_m\gamma^m$ and $r>m$, then
     $$(B+d\gamma^r)A=AB+d\gamma^rA.$$
 \end{enumerate}
\end{lem}
\begin{proof}
 Similar to Lemma \ref{lemfator}.
\end{proof}

\begin{lem}\label{almostfact}
 Let $K|_fF$ be a full extension $F$-almost full relative to $\gamma$. Then:
 \begin{enumerate}[i -]
     \item $K=F[\gamma,K]$;
     \item If $K|_fF$ and $L|_fK$ are almost full then $L|F$ is almost full;
     \item If $L|_fF$ is another full extension and $\pi:K\rightarrow L$ is a full surjective morphism, then $L|_fF$ is $F$-almost full relative to $\pi(\gamma)$;
     \item For all $a_0,...,a_n,b_0,...,b_n\in F$,
     \begin{align*}
         &(a_0+a_1\gamma+a_2\gamma^2+...+a_{n-1}\gamma^{n-1}+a_n\gamma^n)
         (b_0+b_1\gamma+b_2\gamma^2+...+b_{n-1}\gamma^{n-1}+b_n\gamma^n)\subseteq\\
         &a_0b_0+\left(\sum^1_{j=0}a_jb_{1-j}\right)\gamma+...+\left(\sum^{2n-1}_{j=0}a_jb_{(2n-1)-j}\right)\gamma^{2n-1}+\left(\sum^{2n}_{j=0}a_jb_{1-j}\right)\gamma^{2n}
     \end{align*}
     with the convention $a_j=b_j=0$ if $j>n$.
 \end{enumerate}
\end{lem}

Let $K|_fF$ be a full extension and $\alpha\in K$ algebraic over $F$. 
Our aim is to provide an almost full algebraic extension $F(\alpha)|_fF$ containing $F$ and $\alpha$. The key to that is to find a way to describe algebraic elements of $K$. Here we have a first result in this direction.

\begin{teo}[Almost Full Newton's Binom]
Let $K|F$ be an almost full superfield extension $F$-generated by $\{1,\gamma,...,\gamma^n\}$, $\gamma\in K$. Then for all $a,b\in F$,
$$(a+b\gamma)^n=\sum^n_{j=0}\binom{n}{j}a^j(b\gamma)^{n-j}.$$
\end{teo}
\begin{proof}
By induction is enough to prove the case $n=2$. We have
  \begin{align*}
    (a+b\gamma)^2&:=(a+b\gamma)(a+b\gamma)\stackrel{\ref{lemfator3}}{=} a(a+b\gamma)+b\gamma(a+b\gamma)= a^2+ab\gamma+b\gamma a+(b\gamma)^2 \\
    &=a^2+ab\gamma+ab\gamma+(b\gamma)^2=a^2+2ab\gamma+(b\gamma)^2:=\sum^2_{j=0}\binom{n}{j}a^j(b\gamma)^{n-j}.
  \end{align*}
\end{proof}

\begin{teo}\label{teohell3}
 Let $K|F$ be an almost full superfield extension $F$-generated by $\{1,\gamma,...,\gamma^n\}$, $\gamma\in K$. If 
 $$0\in ev(g,(a_0+a_1\gamma+a_2\gamma^2+...+a_{p-1}\gamma^{p-1}+a_p\gamma^p),K)$$
 then there exist $h\in F[X]$ with
 $$0\in ev(h,a_0\gamma+a_1\gamma^2+a_2\gamma^3+...+a_{p-1}\gamma^{p}+a_p\gamma^{p+1}),K).$$
\end{teo}
\begin{proof}
Let $f(X)=\mbox{Irr}(F,\gamma)$ with $f(X)=d_0+d_1X+...+d_nX^n$. We proceed by induction on $p$.

If $p=0$, let $a\in F$. Suppose without loss of generality that $a\ne0$. Consider the set
$$A=d_0+d_1a^{-1}X+d_2a^{-2}X^2+...+d_na^{-n}X^n.$$
Of course, since $1\in a^{-n}a^n$ for all $n\ge1$,
$$d_0+d_1\gamma+...+d_n\gamma^n\subseteq d_0+d_1a^{-1}(a\gamma)+d_2a^{-2}(a\gamma)^2+...+d_na^{-n}(a\gamma)^2$$
Moreover,
$$0\in d_0+d_1\gamma+...+d_n\gamma^n\subseteq d_0+d_1a^{-1}(a\gamma)+d_2a^{-2}(a\gamma)^2+...+d_na^{-n}(a\gamma)^2,$$
which means
$$0\in d_0+d_1a^{-1}(a\gamma)+d_2a^{-2}(a\gamma)^2+...+d_na^{-n}(a\gamma)^2.$$
Then for $i=1,...,n$ there exist $z_i\in d_ia^{-i}$ with
$$0\in d_0+z_1(a\gamma)+z_2(a\gamma)^2+...+z_n(a\gamma)^2.$$
Hence, for $h(X)=d_0+z_1X+...+z_{n-1}X^{n-1}+z_nX^n$ we have $0\in ev(h,a\gamma)$.
\end{proof}

\begin{teo}\label{teohell1}
 Let $K|F$ be an almost full superfield extension $F$-generated by $\{1,\gamma,...,\gamma^n\}$, $\gamma\in K$. Let $n\ge0$, $f(X)=\mbox{Irr}_F[\gamma,K]$ and $a_0,...,a_p\in F$. Then there exists some polynomial $g\in F[X]$ such that 
$$0\in ev(g,(a_0+a_1\gamma+a_2\gamma^2+...+a_{p-1}\gamma^{p-1}+a_p\gamma^p),K).$$
\end{teo}
\begin{proof}
 Let $\gamma$, $f(X)\in F[X]$, with $f(X)=d_0+d_1\gamma+...+d_{n-1}X^{n-1}+d_nX^n$, $p\in\mathbb N$ and $a_0,...,a_p\in F$ in the hypothesis of the Theorem. If $p=0$ there is nothing to prove. for $p=1$ let $a,b\in F$. Then we need to prove that there exist a polynomial  $g(X)\in F[X]$ such that $0\in ev(g,a+b\gamma)$. We can suppose without loss of generality that $b\ne0$. Then
$$0\in d_0+d_1\gamma+...+d_{n-1}\gamma^{n-1}+d_n\gamma^n.$$
Since (using Lemma \ref{lemfator3}) we have
$$d_n(b\gamma)^n\subseteq d_n\left[(a+b\gamma)^n-\sum^{n-1}_{j=0}\binom{n}{j}(b\gamma)^ja^{n-j}\right]
\subseteq -\sum^{n-1}_{j=0}\binom{n}{j}d_n(b\gamma)^ja^{n-j}+d_n(a+b\gamma)^n,$$
we conclude
\begin{align*}
 0&\in d_0+d_1\gamma+...+d_{n-1}\gamma^{n-1}+d_n\gamma^n\Rightarrow \\
 0&\in d_0b^n+d_1b^n\gamma+...+d_{n-1}b^n\gamma^{n-1}+d_nb^n\gamma^n\Rightarrow \\
 0&\in d_0b^n+d_1b^n\gamma+...+d_{n-1}b^n\gamma^{n-1}
 -\sum^{n-1}_{j=0}\binom{n}{j}d_n(b\gamma^j)a^{n-j}+d_n(a+b\gamma)^n\Rightarrow \\
0&\in\left[d_0b^n-\binom{n}{0}d_na^{n}\right]+\left[d_1b^n-\binom{n}{1}d_nba^{n-1}\right]\gamma+...+ 
\left[d_{n-1}b^n-\binom{n}{n-1}d_nb^{n-1}a\right]\gamma^{n-1}+d_n(a+b\gamma)^n.
\end{align*}
Then,
$$0\in e_0+e_1\gamma+...+e_{n-1}\gamma^{n-1}+d_n(a+b\gamma)^n\mbox{ for some }e_j\in d_jb^n-\binom{n}{j}d_nb^ja^{n-j},\,j=0,...,n-1.$$
Repeating this process $n-1$ times more, we arrive at an expression
$$0\in z_0+z_1(a+b\gamma)+...+z_{n-1}(a+b\gamma)^{n-1}+z_n(a+b\gamma)^n$$
for some $z_j\in F$, $j=0,...,n$. Then $0\in ev(g,a+b\gamma)$ for the polynomial
$$g(X)=z_0+z_1X+...+z_{n-1}X^{n-1}+z_nX^n.$$

For the general case, suppose the desired valid for $p$ and let $a_0,...,a_{p+1}\in F$. Then we need to prove that there exist a polynomial 
$g(X)\in F[X]$ such that $0\in ev(g,a_0+a_1\gamma+...+a_{p+1}\gamma^{p+1},K)$. Note that, with Proposition \ref{teohell3}, if $0\in ev(g,a_0+a_1\gamma+...+a_p\gamma^p,K)$ then there exist some $h(X)\in F[X]$ with
$$0\in ev(h,\gamma(a_0+a_1\gamma+...+a_p\gamma^p),K)=ev(h,a_0\gamma+a_1\gamma^2+...+a_p\gamma^{p+1}),K).$$

Now suppose without loss of generality that $a_{p+1}\ne0$. Then by induction hypothesis, there exist $d_0,...,d_m\in F$ with
$$0\in d_0+d_1(a_1+a_2\gamma+...+a_{p+1}\gamma^{p})+...+d_m(a_1+a_2\gamma+...+a_{p+1}\gamma^{p})^m,$$
and by above considerations, there exists $e_0,...,e_s\in F$ with
$$0\in e_0+e_1(a_1\gamma+a_2\gamma^2+...+a_{p+1}\gamma^{p+1})+...+e_s(a_1\gamma+a_2\gamma^2+...+a_{p+1}\gamma^{p+1})^s.$$
Using Lemma \ref{lemfator3} we have
\begin{align*}
    &e_s(a_1\gamma+a_2\gamma^2+...+a_{p+1}\gamma^{p+1})^s\\
    &\subseteq
    e_s\left[(a_0+(a_1\gamma+a_2\gamma^2+...+a_{p+1}\gamma^{p+1}))^s
    -\sum^{s-1}_{j=0}\binom{s}{j}a_0^j(a_1\gamma+a_2\gamma^2+...+a_{p+1}\gamma^{p+1})^{s-j}\right] \\
    &\subseteq
    -\sum^{s-1}_{j=0}\binom{s}{j}e_sa_0^j(a_1\gamma+a_2\gamma^2+...+a_{p+1}\gamma^{p+1})^{s-j}+
    e_s(a_0+a_1\gamma+a_2\gamma^2+...+a_{p+1}\gamma^{p+1})^s,
\end{align*}
we conclude
\begin{align*}
 0&\in e_0+e_1(a_1\gamma+a_2\gamma^2+...+a_{p+1}\gamma^{p+1})+...+e_s(a_1\gamma+a_2\gamma^2+...+a_{p+1}\gamma^{p+1})^s\Rightarrow \\
 0&\in e_0+e_1(a_1\gamma+a_2\gamma^2+...+a_{p+1}\gamma^{p+1})+...+\\
 &-\sum^{s-1}_{j=0}\binom{n}{j}e_sa_0^j(a_1\gamma+a_2\gamma^2+...+a_{p+1}\gamma^{p+1})^{s-j}+
    e_s(a_0+a_1\gamma+a_2\gamma^2+...+a_{p+1}\gamma^{p+1})^s\Rightarrow \\
    0&\in\left[e_0-\binom{s}{0}e_sa^{s}\right]+\left[e_1-\binom{s}{1}e_na^{s-1}\right](a_1\gamma+a_2\gamma^2+...+a_{p+1}\gamma^{p+1})+...+ \\
&+\left[e_{s-1}-\binom{s}{s-1}e_na_0\right](a_1\gamma+a_2\gamma^2+...+a_{p+1}\gamma^{p+1})^{n-1}+
e_s(a_0+a_1\gamma+a_2\gamma^2+...+a_{p+1}\gamma^{p+1})^s.
\end{align*}
Then
\begin{align*}
    0&\in f_0+f_1\gamma+...+f_{s-1}\gamma^{s-1}+e_s(a_0+a_1\gamma+a_2\gamma^2+...+a_{p+1}\gamma^{p+1})^s\mbox{ for some }\\
    f_j&\in e_{s-1}-\binom{s}{s-1}e_na_0,\,j=0,...,n-1
\end{align*}
Repeating this process $n-1$ times more, we arrive at an expression
$$0\in z_0+z_1(a_0+a_1\gamma+a_2\gamma^2+...+a_{p+1}\gamma^{p+1})+...+z_s(a_0+a_1\gamma+a_2\gamma^2+...+a_{p+1}\gamma^{p+1})^s$$
for some $z_s\in F$, $j=0,...,n$. Then $0\in ev(g,a_0+a_1\gamma+a_2\gamma^2+...+a_{p+1}\gamma^{p+1})$ for the polynomial
$$g(X)=z_0+z_1X+...+z_{s-1}X^{s-1}+z_sX^s.$$
\end{proof}

In the sequence, we have a key result, which states that our "best candidate for simple extension", $F(p)$, is an full algebraic and almost full extension of $F$.

\begin{teo}\label{teohell2}
Let $F$ be a superfield and $p\in F[X]$ be an irreducible polynomial. Then $F(p)$ is a $F$-almost full extension and $F(p)|_fF$ is algebraic. In particular, for all $p\in F[X]$ there exist an algebraic strong extension $K|_fF$ such that $p$ has a root in $K$.
\end{teo}
\begin{proof}
Let $\omega=\overline{X}\in F(p)$. Then $F(p)$ is generated by $\{1,\omega,...,\omega^{n-1}\}$, with $n=\deg(p)$ and since the operations in $F(p)$ are inherited from $F[X]$, we get that $F(p)$ is almost full. By construction, for all $a_0,...,a_{n-1}\in F$ we have
$$[a_0+a_1X...a_{n-1}X^{n-1}]=a_0+a_1\omega+...+a_{n-1}\omega^{n-1},$$
in other words, is for all $a_0,...,a_{n-1}\in F$ the set $a_0+a_1\omega+...+a_{n-1}\omega^{n-1}$ is unitary.

Now let $\sigma\in F(p)$, $\sigma=a_0+a_1\omega+...+a_{n-1}\omega^{n-1}$.
Using Theorem \ref{teohell1} there is some $g\in F[X]$ such that
$$0\in ev(g,(a_0+a_1\omega+...+a_{n-1}\omega^{n-1}),F(p))=ev(g,\sigma,F(p)).$$
Then $\sigma$ is algebraic $F(p)$-algebraic over $F$, completing the proof.
\end{proof}

Keeping on hands the Theorem \ref{teohell2}, we work in order to legitimate $F(p)$ as the simple extension of $F$ by $\alpha$. But before we do that, lets make some considerations about general almost full extensions.

The proof of Theorem \ref{teohell2} strongly rely in the fact that $a_0+a_1\omega+...+a_{n-1}\omega^{n-1}$ is unitary. It is a special property of $F(p)$, and is not necessarily valid for a general almost full full extension.

For an almost full extension $K|_fF$ denote
$$\mbox{Alg}(K,F)=\{\alpha\in K:\alpha\mbox{ is algebraic over }F\}.$$
We do not know if $\mbox{Alg}(K,F)$ is a superfield in general. The difficult here is that despite the fact that Theorem \ref{teohell2} is still available, we cannot use it to conclude that all elements in $\alpha\beta$ and $\alpha+\beta$ are algebraic if $\alpha$ and $\beta$ are algebraic.

It is time to define a notion of simple extension.
\begin{defn}[Simple Extension]
Let $K|_fF$ be a full extension and $\alpha\in K$ algebraic. We define the \textbf{simple extension} $F(\alpha,K)$ by
$$F(\alpha,K):=\bigcap\{L:L|_fF\mbox{ is full and }F[\alpha]\subseteq L\}.$$
Note that we have a full extension $F(\alpha,K)|_fF$. If $\lambda_1,...,\lambda_n\in K$ are algebraic, we define
$$F(\lambda_1,...,\lambda_n,K):=F(\lambda_1,...,\lambda_{n-1},K)(\lambda_n,K).$$
By Theorem \ref{unicityext} we can simply write $F(\alpha)$ to indicate $F(\alpha,K)$
\end{defn}

\begin{teo}
$ $
\begin{enumerate}[i -]
    \item Let $K|_fF$ be a full extension with $\alpha\in K$ algebraic. Let $p(X)=\mbox{Irr}_F(\alpha,K)$. Then $F(\alpha)\cong F(p)$.
    
    \item Let $K|_fF$ be a full extension and $\alpha,\beta\in K$ algebraic such that $F(\alpha)(\beta)|_fF(\alpha)$ and $F(\beta)(\alpha)|_fF(\beta)$ are almost full extensions relative to $\alpha$ and $\beta$ respectively. Then
    $$F(\alpha)(\beta)\cong F(\beta)(\alpha).$$
    
    \item Let $K|_fF$ be a full extension. For all $\alpha_1,...,\alpha_n\in K$ and all $\sigma\in S_n$ we have
    $$F(\alpha_1,...,\alpha_n)\cong F(\alpha_{\sigma(1)},...,\alpha_{\sigma(n)}).$$
\end{enumerate}
\end{teo}
\begin{proof}
$ $
\begin{enumerate}[i -]
    \item  We have that $F(p)|_fF$ is a full extension containing $F[\alpha,K]$ (see Theorem \ref{root1}), so $F(\alpha)\subseteq F(p)$. Moreover, $F(p)$ is generated by $\{1,\alpha,...,\alpha^{n-1}\}$, where $n=\deg(p)$. Then $F[\alpha]=F[\{1,\alpha,...,\alpha^{n-1}\}]$ already is a superfield and
    $$F(p)\cong F[\{1,\alpha,...,\alpha^{n-1}\}]=F(\alpha).$$
    
    \item By construction, $ev(p,\alpha,F(\alpha)[X])\subseteq F(\beta)(\alpha)$ for all $p\in F(\alpha)[X]$. Then $F(\alpha)(\beta)\subseteq F(\beta)(\alpha)$. Reverting the argument we conclude $F(\beta)(\alpha)\subseteq F(\alpha)(\beta)$.
    
    \item Just use previous item and induction.
\end{enumerate}
\end{proof}

\begin{cor}
Let $K|_fF$ be a full extension with $\alpha\in K$ algebraic and $\deg(\mbox{Irr}_F(\alpha))=n$. Then
$$F(\alpha)\cong\{a_0+a_1\alpha+...+a_n\alpha^n:a_0,...,a_n\in F\},$$
with operations in the set on the right inherited from $F[X]$.
\end{cor}

Of course, deal with $F(p)$ is much easier to deal with the general expression
$$\bigcap\{L:L|_fF\mbox{ is full and }F[\alpha]\subseteq L\}$$
in the sense of make calculations. But the task of determining $F(p)$ "by hand" was already difficult in the field case. In the superfield case this difficult is accentuate, even for low degree polynomials.

\begin{ex}[Quadratic Extensions of $H_3$]\label{ext1ex}
Of course, the only irreducible polynomial of degree 2 over $H_2$ is $f(X)=X^2+2$. We want to describe some possibilities for $H_3(\sqrt{2},K)$ (even in the case of non full extensions).

We first use Theorem \ref{root1}. Let $\mbox{Irr}_{H_3}(\sqrt{2})=p(X)=X^2+2$ and consider $K=H_3(p)$. Lets look closely at the operations on $K$. Denote an element in $K$ by $[f]\in K$, $f\in H_3[X]$. We have
$$K=\{[0],[1],[2],[X],[2X],[1+X],[2+X], [1+2X], [2+2X]\}.$$
By definition, for $[f],[g]\in K$ we have
$$[f]+[g]:=\{[h]:h\in f+g\}\mbox{ and }[f]\cdot[g]:=\{[h]:h\in fg\}.$$
With these rules is easy to show that $K|H_3$ is an algebraic full extension (for example, $[1+X]$ is a root of $f(X)=X^2+1$). In fact, $K=H_3(\sqrt2)$. Moreover $K$ is not a hyperfield because
$$([1+X])([1+X])=\dot K.$$

Now le $L=H_3\times_hH_5$. Note that $|L|=(3-1)(5-1)+1=2\cdot4+1=9$. Moreover, we have a morphism $i:H_3\hookrightarrow H_5$ given by the rule $i(x)=(1,x^2)$. Denoting $\omega=(1,2)$, we have
$$\omega^2=(1,2)^2=(1,2)\cdot(1,2)=(1,2^2)=(1,4)=i(2).$$
More explicitly, doing the following identifications
\begin{align*}
    (1,1)\mapsto1,&\qquad\qquad(2,1)\mapsto a,\\
    (1,2)\mapsto\omega,&\qquad\qquad(2,2)\mapsto b,\\
    (1,3)\mapsto2\omega,&\qquad\qquad(2,3)\mapsto c,\\
    (1,4)\mapsto2,&\qquad\qquad(2,4)\mapsto d,
\end{align*}
we have that
$$L\cong\{0,1,2,\omega,2\omega,a,b,c,d\}$$
with the following table of operations:
\begin{center}
    \begin{tabular}{|c||c|c|c|c|c|c|c|}
    \hline
   $+$ & $\omega$ & $2\omega$ & $a$ & $b$ & $c$ & $d$ \\
    \hline
    \hline
    $1$ & $\{1,\omega,a,b\}$ & $\{1,2\omega,a,c\}$ & $K\setminus\{0\}$ & $\{1,\omega,a,b\}$ & $\{1,2\omega,a,c\}$ & $\{1,2,a,d\}$ \\
    \hline
    $2$ & $\{2,\omega,b,d\}$ & $\{2,2\omega,c,d\}$ & $\{1,2,a,d\}$ & $\{2,\omega,b,d\}$ & $\{2,2\omega,c,d\}$ & $K\setminus\{0\}$ \\
    \hline
    $\omega$ & $K$ & $\{\omega,2\omega,b,c\}$ & $\{1,\omega,a,b\}$ & $K\setminus\{0\}$ & $\{\omega,2\omega,b,c\}$ & $\{2,\omega,b,d\}$ \\
    \hline
    $2\omega$ & $\{\omega,2\omega,b,c\}$ & $K$ & $\{1,2\omega,a,c\}$ & $\{\omega,2\omega,b,c\}$ & $K\setminus\{0\}$ & $\{2,2\omega,c,d\}$ \\
    \hline
    $a$ & $\{1,\omega,a,b\}$ & $\{1,2\omega,a,c\}$ & $K$ & $\{1,\omega,a,b\}$ & $\{1,2\omega,a,c\}$ & $\{1,2,a,d\}$ \\
    \hline
    $b$ & $K\setminus\{0\}$ & $\{\omega,2\omega,b,c\}$ &  $\{1,\omega,a,b\}$ & $K$ & $\{\omega,2\omega,b,c\}$ & $\{2,\omega,b,d\}$ \\
    \hline
    $c$ & $\{\omega,2\omega,b,c\}$ & $K\setminus\{0\}$ & $\{1,2\omega,a,c\}$ & $\{\omega,2\omega,b,c\}$ & $K$ & $\{2,2\omega,c,d\}$ \\
    \hline
    $d$ & $\{2,\omega,b,d\}$ & $\{2,2\omega,c,d\}$ & $\{1,2,a,d\}$ & $\{2,\omega,b,d\}$ & $\{2,2\omega,c,d\}$ & $K$ \\
    \hline
\end{tabular}
\end{center}

\begin{center}
    \begin{tabular}{|c||c|c|c|c|c|c|c|c|c|}
    \hline
   $\cdot$ & $2$ & $\omega$ & $2\omega$ & $a$ & $b$ & $c$ & $d$ \\
    \hline
    \hline
    $2$ & $1$ & $2\omega$ & $\omega$ & $d$ & $c$ & $b$ & $a$ \\
    \hline
    $\omega$ & $2\omega$ & $2$ & $1$ & $b$ & $d$ & $a$ & $c$ \\
    \hline
    $2\omega$ & $\omega$ & $1$ & $2$ & $c$ & $a$ & $d$ & $b$ \\
    \hline
    $a$ & $d$ & $b$ & $c$ & $1$ & $\omega$ & $2\omega$ & $2$ \\
    \hline
    $b$ & $c$ & $d$ & $a$ & $\omega$ & $2$ & $1$ & $2\omega$ \\
    \hline
    $c$  & $b$ & $a$ & $d$ & $2\omega$ & $1$ & $2$ & $\omega$ \\
    \hline
    $d$  & $a$ & $c$ & $b$ & $2$ & $2\omega$ & $\omega$ & $1$ \\
    \hline
\end{tabular}
\end{center}
and of course, $1+1=2+2=L$, $0+x=\{x\}$, $1\cdot x=x$ and $0\cdot x=0$ for all $x\in L$. With these calculations we immediately have that $L$ is an algebraic extension of $H_3$. 

Now Let $q$ be an odd prime integer greater than 3. The same calculations (with $\omega=(1,2)$) proves that $H_3\times_hH_q$ is another algebraic extension of $H_3$. Of course, we clearly have $H_3\times_hH_5\ncong H_3\times_hH_q$ for $g\ge7$. And since all these $H_3\times_hH_q$ are hyperfields and $K$ is a superfield that is not a hyperfield we have $K\ncong H_3\times_hH_q$ for all prime $q\ge5$. Conclusion: we have infinite non isomorphic algebraic (and non full) hyperfield extensions of $H_3$.
\end{ex}

\section{Algebraic Closure}

As expected, there are some generalizations to the classic notion of algebraic closure for fields.

 \begin{defn}[Algebraic Closures]\label{extension}
Let $F$ and $K$ be superfields.
\begin{enumerate}[i -]
    \item We say that $K$ is a \textbf{proto algebraic closure} of $F$ if $K$ is algebraically closed and $K|_pF$ is algebraic.
    \item We say that $K$ is an \textbf{algebraic closure} of $F$ if $K$ is algebraically closed and $K|F$ is algebraic.
    \item We say that $K$ is a \textbf{full algebraic closure} of $F$ if $K$ is algebraically closed and $K|_fF$ is algebraic.
\end{enumerate}
Of course, all these notions coincide if we choose a field $F$.
\end{defn}

\begin{lem}\label{lemuinq1}
 Let $F$ be a superfield and $K|_fF$ be an algebraic extension. If $K$ is a full algebraic closure of $F$ then $K|_fF$ is a maximal full algebraic extension.
\end{lem}
\begin{proof}
 If $K|_fF$ is not maximal, there is a nontrivial full algebraic extension $L|_fK$. In particular, there is a nontrivial simple extension $K(\alpha)|_fK$, then $K$ is not an algebraic closure.
\end{proof}

Here we achieve the main result of this present paper.

\begin{teo}[Existence of the full Algebraic Closure]\label{algclos}
Let $F$ be a superfield. Then exists a full superfield extension $K|_fF$ such that $K$ is algebraically closed (and then, a full algebraic closure of $F$). Moreover, we can choose $K$ in order that $K|_fF$ is algebraic.
\end{teo}
\begin{proof}
Let $F$ be a superfield. Consider the following set
$$A:=\{\omega^f_i:f\in F[X],\,\deg(f)\ge1,\,i=1,...,\deg(f)\}.$$
In other words, for each $f$ of degree greater or equal to $1$, we are choosing elements $\omega^f_1,...,\omega^f_{\deg(f)}$ 
to represent "some possible roots for $f$". For each $a\in F$, $a$ is the root of $f_a(X)=X-a$, and hence there is an element 
$\omega^{f_a}_1\in A$. Let
$$\Omega=\left(\mathcal P(A)\setminus\bigcup_{a\in F}\{\omega^{f_a}_1\}\right)\cup F.$$
Then $F\subseteq\Omega$. Now, consider all the possible superfields that can be defined on elements of $\Omega$. Denote the set 
of all such superfields by $\mathcal E$. Since $\mathcal E\subseteq\Omega$, it is in fact a set, and since $F\in\mathcal E$, it is 
a non-empty set.

Let $E|_fF$ be an almost full algebraic extension of $F$-generated by $\{1,\gamma,...,\gamma^n\}$ where $\gamma\in E\setminus F$ is a root of $f$ in $F[X]$. In other words, we have $E=F(\gamma)$. Let $\omega\in\Omega\setminus F$. We can "make the variable change" 
$\gamma\mapsto\omega$ and choose distinct elements for all elements in $F(\gamma)$ in order to get a field $F(\omega)\cong 
F(\gamma)$, such that $F\subseteq F(\omega)\subseteq\Omega$.

Then, for all almost full algebraic extension $E_j\subseteq\Omega$ obtained by the above process, we can take the set
$$S=\{E_j:j\in J\}.$$
We have $F\in S$ and $S$ is partially ordered by inclusion.

Let $T=\{E_{kj}:k\in K\}$ be a chain in $S$ and
$$W=\bigcup_{k\in K}E_{kj}.$$
Since $W$ is an algebraic extension of $F$, we get $W\in S$. By Zorn's Lemma, there exist some maximal element $\overline 
F\in S$. We prove that $\overline F$ is an algebraic closure of $F$.

In fact, suppose that exists $f(X)\in F[X]$ such that $f$ has no roots in $\overline F[X]$. Then, take $\omega\in\Omega$ such 
that $\omega\notin\overline F$ and $\omega$ is a root of $f(X)$. Consider the field $\overline F(\omega)$ as we did above. 
Then $\overline F(\omega)$ is an algebraic extension with $\overline F\subsetneq\overline F(\omega)$, contradicting the 
maximality of $\overline F$, which complete the proof.
\end{proof}

We are surprisingly able to prove the uniqueness of full algebraic closures.

\begin{teo}[Uniqueness of the full Algebraic Closure]\label{uniq}
Let $F$ be a superfield. Let $K_1,K_2$ be two full algebraic closures of $F$. Then $K_1\cong K_2$.
\end{teo}

To prove Theorem \ref{uniq} we need two Lemmas. Let $L|_fF$ be a full superfield extension and $N$ be another superfield. An \textbf{$F$-embedding} is a full embedding $\iota:L\rightarrow N$ such that $\iota(a)=a, a \in F$.

\begin{lem}\label{uniq1}
Let $L|_fF$ be an algebraic full extension and $N|_fL$ another algebraic full extension, and $\overline F$ some full algebraic closure of $F$. There is a $F$-embedding $i:L\rightarrow\overline F$ and once $i$ is picked there exists a $F$-embedding $N\rightarrow\overline F$ extending $i$.
\end{lem}
\begin{proof}
Since a full embedding $i:L\rightarrow\overline F$ realizes the full algebraically closed $\overline F$ as an algebraic extension of $L$ (and hence as a full algebraic closure of $L$), by renaming the base superfield as $L$ it suffices to just prove the first part: any strong algebraic extension admits a full embedding into a specified full algebraic closure.

Let $\Sigma$ to be the set of pairs $(K,i)$ such that $K|_fF$, $L|_fK$ and the inclusion map $i:K\rightarrow\overline F$ is a $F$-embedding. Of course, $(F,i)\in\Sigma$, and using the partial order defined by
$$(K_1,i_1)\le(K_2,i_2) \mbox{ iff }K_2|_fK_1,\,L|_fK_2\mbox{ and }i_2|_{k_1}=i_1,$$
we obtain that every chain has an upper bound (the superfield obtained by directed union). Then we are under the hypothesis of Zorn's Lemma and there exists a maximal element $(N,i)\in\Sigma$.

We just have to show $N=L$. Pick $\alpha\in L$, so $\alpha$ is algebraic over $N$ (as it is algebraic over $F$). We have $N(\alpha)|_fN$ and $\overline F|_fN(\alpha)$. In other words, the inclusion map $i:N(\alpha)\rightarrow\overline F$ is a full $N$-embedding. By maximality of $N$ we get $N(\alpha)=N$ for all $\alpha\in L$, which imply $N=L$.
\end{proof}

\begin{lem}\label{uniq2}
 Let $F$ be a superfield and $\overline{F}$ be some full algebraic closure of $F$. If $\phi:\overline{F}\rightarrow\overline F$ is a $F$-embedding then $\phi$ is an isomorphism.
\end{lem}
\begin{proof}
 We only need to show that $\phi$ is surjective. Let $\gamma\in\overline F$. Then there exist $p(X)\in F[X]$, saying $p(X)=X^n+a_{n-1}X^{n-1}+...+a_1X+a_0$ with $0\in p(\gamma)$. Since $\phi$ is a $F$-embedding, we have
 $$p^\phi(X):=X^n+\phi(a_{n-1})X^{n-1}+...+\phi(a_1)X+\phi(a_0)=X^n+a_{n-1}X^{n-1}+...+a_1X+a_0=p(X).$$
 Then $\phi(\gamma)$ is a root of $p(X)$ because
 \begin{align*}
   &0\in a_n\gamma^n+a_{n-1}\gamma^{n-1}+...+a_1\gamma+a_0\Rightarrow
   \phi(0)\in \phi(a_n\gamma^n+a_{n-1}\gamma^{n-1}+...+a_1\gamma+a_0)\Rightarrow \\
   &0\in a_n\phi(\gamma)^n+a_{n-1}\phi(\gamma)^{n-1}+...+a_1\phi(\gamma)+a_0.
 \end{align*}
 
 Since $\phi$ is a full embedding, we have a full embedding $\phi(\overline{F})\hookrightarrow\overline{F}$. Then $\overline{F}|_f\phi(\overline{F})$. Since $\overline F$ is algebraically closed, every non-constant polynomial $p(X)\in F[X]$ has a root $\gamma\in\overline F$, and then, a root $\phi(\gamma)\in \phi(\overline{F})$. If $\phi(\overline{F})\ne \overline F$, we have a contradiction with the maximality of $\phi(\overline{F})$ obtained in Lemma \ref{lemuinq1}.
\end{proof}

\begin{proof}[Proof of Theorem \ref{uniq}]
By Lemma \ref{uniq1} applied to $L=K_1$ and $\overline F=K_2$ (a full algebraic closed superfield equipped with a structure of algebraic extension of $F$), there exists a $F$-embedding $i_1:K_1\rightarrow K_2$. By the very same argument, there also exists a $F$-embedding $i_2:K_2\rightarrow K_1$. Moreover, $i_1\circ i_2:K_1\rightarrow K_1$ and $i_2\circ i_1:K_2\rightarrow K_2$ are $F$-embeddings. By Lemma \ref{uniq2}, both $i_1\circ i_2$ and $i_2\circ i_1$ are isomorphisms, implying that $i_1$ and $i_2$ are also isomorphisms.
\end{proof}

\begin{ex}\label{ext2ex}
Lets look at $H_3$ again. Consider $L_1=H_3\times_hH_5$ and $L_2=H_3\times_hH_7$. We do not know precisely the relations between the full algebraic closures $\overline{H_3}$, $\overline{L_1}$ and $\overline{L_2}$.

Of course, since $L_1|H_3$ and $L_2|H_3$ are algebraic extensions of $H_3$, we have that $\overline{L_1}$ and $\overline{L_2}$ are algebraic closures of $\overline{H_3}$. Since $L_2$ is an algebraic extension of $L_1$, we know that $\overline{L_2}$ is an algebraic closure of $L_1$. But we do not know if $\overline{H_3}$, $\overline{L_1}$ and $\overline{L_2}$ are isomorphic (or not). 
\end{ex}

\bibliographystyle{plain}
\bibliography{one_for_all.bib}

\begin{thebibliography}{10}

\bibitem{al2019some}
Madeline Al~Tahan, Sarka Hoskova-Mayerova, and Bijan Davvaz.
\newblock Some results on (generalized) fuzzy multi-hv-ideals of hv-rings.
\newblock {\em Symmetry}, 11(11):1376, 2019.

\bibitem{ameri2020advanced}
R~Ameri, M~Eyvazi, and S~Hoskova-Mayerova.
\newblock Advanced results in enumeration of hyperfields.
\newblock {\em Aims Mathematics}, 5(6):6552--6579, 2020.

\bibitem{ameri2017multiplicative}
R~Ameri, A~Kordi, and S~Hoskova-Mayerova.
\newblock Multiplicative hyperring of fractions and coprime hyperideals.
\newblock {\em Analele Universitatii" Ovidius" Constanta-Seria Matematica},
  25(1):5--23, 2017.

\bibitem{ameri2019superring}
Reza Ameri, Mansour Eyvazi, and Sarka Hoskova-Mayerova.
\newblock Superring of polynomials over a hyperring.
\newblock {\em Mathematics}, 7(10):902, 2019.

\bibitem{baker2021structure}
Matthew Baker and Tong Jin.
\newblock On the structure of hyperfields obtained as quotients of fields.
\newblock {\em Proceedings of the American Mathematical Society},
  149(1):63--70, 2021.

\bibitem{baker2021descartes}
Matthew Baker and Oliver Lorscheid.
\newblock Descartes' rule of signs, newton polygons, and polynomials over
  hyperfields.
\newblock {\em Journal of Algebra}, 569:416--441, 2021.

\bibitem{bowler2021classification}
Nathan Bowler and Ting Su.
\newblock Classification of doubly distributive skew hyperfields and stringent
  hypergroups.
\newblock {\em Journal of Algebra}, 574:669--698, 2021.

\bibitem{davvaz2016codes}
B~Davvaz and T~Musavi.
\newblock Codes over hyperrings.
\newblock {\em Matematicki Vesnik}, 68(1):26--38, 2016.

\bibitem{roberto2021hauptsatz}
Kaique~Matias de~Andrade~Roberto, Hugo~Rafael de~Oliveira~Ribeiro, and
  Hugo~Luiz Mariano.
\newblock The {A}ranson-{P}fister {H}auptsatz for special hyperfields.
\newblock {\em Preliminary version in https://arxiv.org/abs/2210.03784}, 2022.

\bibitem{roberto2021quadratic}
Kaique~Matias de~Andrade~Roberto, Hugo~Rafael de~Oliveira~Ribeiro, and
  Hugo~Luiz Mariano.
\newblock Quadratic structures associated to (multi) rings.
\newblock {\em Categories and General Algebraic Structures}, 16(1):105--141,
  2022.

\bibitem{roberto2021ktheory}
Kaique~Matias de~Andrade~Roberto and Hugo~Luiz Mariano.
\newblock K-theories and free inductive graded rings in abstract quadratic
  forms theories.
\newblock {\em Categories and General Algebraic Structures}, 17(1):1--46, 2022.

\bibitem{roberto2021superrings}
Kaique~Matias de~Andrade~Roberto and Hugo~Luiz Mariano.
\newblock On superrings of polynomials and algebraically closed multifields.
\newblock {\em Journal of Applied Logics}, 9(1):419--444, 2022.

\bibitem{ribeiro2021anel}
Hugo~Rafael de~Oliveira~Ribeiro.
\newblock {\em Anel de Witt para semigrupos reais, envolt{\'o}ria von Neumann e
  B-pares}.
\newblock PhD thesis, Universidade de S{\~a}o Paulo, 2021.

\bibitem{ribeiro2016functorial}
Hugo~Rafael de~Oliveira~Ribeiro, Kaique~Matias de~Andrade~Roberto, and
  Hugo~Luiz Mariano.
\newblock Functorial relationship between multirings and the various abstract
  theories of quadratic forms.
\newblock {\em S{\~a}o Paulo Journal of Mathematical Sciences}, 16:5--42, 2022.

\bibitem{eppolito2020hopf}
Chris Eppolito, Jaiung Jun, and Matt Szczesny.
\newblock Hopf algebras for matroids over hyperfields.
\newblock {\em Journal of Algebra}, 556:806--835, 2020.

\bibitem{worytkiewiczwitt2020witt}
Pawel Gladki and Krzysztof Worytkiewicz.
\newblock Witt rings of quadratically presentable fields.
\newblock {\em Categories and General Algebraic Structures}, 12(1):1--23, 2020.

\bibitem{golzio2018brief}
Ana~Claudia Golzio.
\newblock A brief historical survey on hyperstructures in algebra and logic.
\newblock {\em South American Journal of Logic}, 2018.

\bibitem{jun2015algebraic}
Jaiung Jun.
\newblock Algebraic geometry over hyperrings.
\newblock {\em Advances in Mathematics}, 323:142--192, 2018.

\bibitem{jun2018valuations}
Jaiung Jun.
\newblock Valuations of semirings.
\newblock {\em Journal of Pure and Applied Algebra}, 222(8):2063--2088, 2018.

\bibitem{jun2021geometry}
Jaiung Jun.
\newblock Geometry of hyperfields.
\newblock {\em Journal of Algebra}, 569:220--257, 2021.

\bibitem{marshall2006real}
Murray Marshall.
\newblock Real reduced multirings and multifields.
\newblock {\em Journal of Pure and Applied Algebra}, 205(2):452--468, 2006.

\bibitem{massouros1985theory}
Ch~G Massouros.
\newblock Theory of hyperrings and hyperfields.
\newblock {\em Algebra and Logic}, 24(6):477--485, 1985.

\bibitem{massouros2009join}
Christos~G Massouros and Gerasimos~G Massouros.
\newblock On join hyperrings.
\newblock In {\em Proceedings of the 10th International Congress on Algebraic
  Hyperstructures and Applications, Brno, Czech Republic}, pages 203--215,
  2009.

\bibitem{massouros1999homomorphic}
Geronimos~G Massouros and Christos~G Massouros.
\newblock Homomorphic relation on hyperingoinds and join hyperrings.
\newblock {\em Ratio Mathematica}, 13(1):61--70, 1999.

\bibitem{nakassis1988recent}
Anastase Nakassis.
\newblock Recent results in hyperring and hyperfield theory.
\newblock {\em International Journal of Mathematics and Mathematical Sciences},
  11, 1988.

\bibitem{pelea2006multialgebras}
Cosmin Pelea and Ioan Purdea.
\newblock Multialgebras, universal algebras and identities.
\newblock {\em Journal of the Australian Mathematical Society}, 81(1):121--140,
  2006.

\bibitem{viro2010hyperfields}
Oleg Viro.
\newblock Hyperfields for tropical geometry i. hyperfields and dequantization.
\newblock {\em arXiv preprint arXiv:1006.3034}, 2010.

\end{thebibliography}
\end{document}